\documentclass{article}
\usepackage{fullpage}

\usepackage[utf8]{inputenc}
\usepackage{amsthm,amssymb}
\usepackage{amsmath}
\usepackage{thmtools,mathtools}
\usepackage{caption,subcaption}
\usepackage{epsfig,graphicx,graphics,color,tikz,tikz-cd,tcolorbox}
\usepackage{xcolor,colortbl,hhline,xcolor}
\usetikzlibrary{calc,intersections,through, shapes.geometric, positioning, math, decorations.markings,chains, scopes, arrows.meta,tikzmark,decorations.pathreplacing,calligraphy}
\usepackage{tikzsymbols}
\usepackage{pstricks-add}
\usepackage{calc}
\usepackage{enumerate,enumitem}
\usepackage[sort]{cite}
\usepackage{xspace}
\usepackage{bm}
\usepackage{csquotes}
\usepackage{hyperref}
\hypersetup{
    colorlinks=true,       %
    linkcolor=blue,        %
    citecolor=magenta,         %
    filecolor=magenta,     %
    urlcolor=cyan,         %
    linktoc=all
}
\usepackage{cleveref}
\usepackage{titling}

\usepackage{tcolorbox}
\newtcolorbox{probbox}{arc=6pt,
                      colback=white!100,
                      colframe=black!50,
                      before skip=6pt,
                      after skip=6pt,
                      boxsep=1pt,
                      left=6pt,
                      right=6pt,
                      top=4pt,
                      bottom=4pt}

\theoremstyle{plain}
\newtheorem{thm}{Theorem}[section]
\newtheorem{lem}[thm]{Lemma}
\newtheorem{cor}[thm]{Corollary}
\newtheorem{cl}[thm]{Claim}

\newtheorem{conj}[thm]{Conjecture}

\theoremstyle{definition}

\newtheorem{rem}[thm]{Remark}

\newcommand{\cB}{\mathcal{B}}

\newcommand{\cE}{\mathcal{E}}

\newcommand{\cH}{\mathcal{H}}

\newcommand{\underlying}[1]{u(#1)}

\newlength{\bibitemsep}
\newlength{\bibparskip}
\let\oldthebibliography\thebibliography
\renewcommand\thebibliography[1]{%
  \oldthebibliography{#1}%
  \setlength{\parskip}{\bibitemsep}%
  \setlength{\itemsep}{\bibparskip}%
}

\makeatletter
\renewcommand{\paragraph}{%
  \@startsection{paragraph}{4}%
  {\z@}{1.6ex \@plus 1ex \@minus .2ex}{-0.5em}%
  {\normalfont\normalsize\bfseries}%
}
\makeatother

\def\final{1}  %
\def\iflong{\iffalse}
\ifnum\final=0  %
\usepackage[dvipsnames]{xcolor}
\newcommand{\yu}[1]{{\color{magenta}[{\tiny \textbf{Yu:} \bf #1}]\marginpar{\color{red}*}}}
\newcommand{\yutaro}[1]{{\color{red}[{\tiny \textbf{Yutaro:} \bf #1}]\marginpar{\color{red}*}}}
\newcommand{\tamas}[1]{{\color{red}[{\tiny \textbf{Tam\'as:} \bf #1}]\marginpar{\color{red}*}}}
\newcommand{\kristof}[1]{{\color{red}[{\tiny \textbf{Krist\'of:} \bf #1}]\marginpar{\color{red}*}}}
\else %
\newcommand{\yu}[1]{}
\newcommand{\yutaro}[1]{}
\newcommand{\tamas}[1]{}
\newcommand{\kristof}[1]{}
\fi

\title{Rainbow Arborescence Conjecture}

\thanksmarkseries{alph}
\author{
Kristóf Bérczi\thanks{MTA-ELTE Matroid Optimization Research Group and HUN-REN–ELTE Egerváry Research Group, Department of Operations Research, ELTE Eötvös Loránd University, and HUN-REN Alfréd Rényi Institute of Mathematics, Budapest, Hungary. Email: \texttt{kristof.berczi@ttk.elte.hu}.}
\and
Tamás Király\thanks{HUN-REN-ELTE Egerv\'ary Research Group, Department of Operations Research, E\"otv\"os Loránd University, Budapest, Hungary. Email: \texttt{tamas.kiraly@ttk.elte.hu}.}
\and
Yutaro Yamaguchi\thanks{Department of Information and Physical Sciences, Graduate School of Information Science and Technology, Osaka University, Osaka, Japan. Email: \texttt{yutaro.yamaguchi@ist.osaka-u.ac.jp}.}
\and
Yu Yokoi\thanks{Department of Mathematical and Computing Science, School of Computing, Institute of Science Tokyo, Tokyo, Japan. Email: \texttt{yokoi@comp.isct.ac.jp}.}
}

\date{}

\begin{document}
\maketitle
\thispagestyle{empty}

\begin{abstract}
The famous Ryser--Brualdi--Stein conjecture asserts that every $k \times k$ Latin square contains a partial transversal of size $k-1$. Since its appearance, the conjecture has attracted significant interest, leading to several proposed generalizations. One of the most notable of these, by Aharoni, Kotlar, and Ziv, conjectures that $k$ disjoint common bases of two matroids of rank $k$ have a common independent partial transversal of size $k-1$. Although simple counterexamples show that the size $k-1$ above cannot be improved to $k$ (i.e., a transversal instead of a partial transversal), it is remarkable that no such counterexample is known for the special case of spanning arborescences. This motivated the formulation of the \emph{Rainbow Arborescence Conjecture}: any graph on $n$ vertices formed by the union of $n-1$ spanning arborescences contains an arborescence using exactly one arc from each. 

We prove several partial results on this conjecture. We show that the computational problem of testing the existence of such an arborescence with a fixed root is NP-complete, verify the conjecture in several special cases, and study relaxations of the problem. In particular, we establish the validity of the conjecture when the underlying undirected graph is a cycle; this also yields a new result on systems of distinct representatives for intervals on a cycle. 
\medskip

\noindent \textbf{Keywords:} Latin squares, Matroid intersection, Rainbow arborescences, Transversals
\end{abstract}
\newpage
\pagenumbering{roman}
\tableofcontents
\thispagestyle{empty}
\newpage
\pagenumbering{arabic}
\setcounter{page}{1}

\section{Introduction}
\label{sec:intro}

A Latin square is a $k \times k$ matrix where each row/column contains all distinct numbers, $1$ through $k$.
A partial transversal is a subset of entries in the matrix that are distinct numbers in distinct rows and columns.
Ryser, Brualdi, and Stein~\cite{Brualdi_Ryser_1991,stein1975transversals} conjectured that every $k \times k$ Latin square contains a partial transversal of size $k-1$.
Woolbright~\cite{woolbright1978n} and independently Brouwer, de Vries, and Wieringa~\cite{brouwer1978lower} showed that a partial transversal of size $k - \sqrt{k}$ always exists.
Shor and Hatami~\cite{shor1982lower,hatami2008lower} improved this bound to $k - O(\log^2 k)$.
Cameron and Wanless~\cite{cameron2005covering} further relaxed the problem, showing that every Latin square contains a subset of entries such that they are in distinct rows and columns and no number appears more than twice.

A matroidal generalization of the conjecture was proposed by Chappell~\cite{chappell1999matroid} and later by Kotlar and Ziv~\cite{kotlar2012length}.
A matroidal Latin square is a $k \times k$ matrix whose elements are members of the ground set of a matroid, where each row/column forms an independent set.
An independent partial transversal is an independent set comprising entries from distinct rows and columns.
As a matroidal analogue of the Ryser--Brualdi--Stein conjecture, Kotlar and Ziv~\cite{kotlar2012length} conjectured that every $k \times k$ matroidal Latin square has an independent partial transversal of size $k-1$, and proved a lower bound of $\lceil 2k / 3 \rceil$ for the maximum size of such a set.

Aharoni, Kotlar, and Ziv~\cite{aharoni2013independent,aharoni2016rainbow} further generalized the conjecture to matroid intersection as follows: \emph{If $M$ and $N$ are matroids on the same ground set, then for any $k$ pairwise disjoint common independent sets of size $k$, there exists a common independent set of size $k-1$ intersecting each of the $k$ sets in at most one element.} They also established a lower bound $k-\sqrt{k}$ for the size of such a set, extending the results of \cite{woolbright1978n,brouwer1978lower} and improving the bound in \cite{kotlar2012length}. 

A simple example shows that the value $k-1$ above cannot be improved to $k$: the two perfect matchings of $C_4$ do not have a transversal that is a perfect matching. Remarkably, no such example is known in the case where the common bases are the spanning arborescences of a directed graph $G$. Here, an \emph{arborescence} in $G$ is a sub-digraph whose underlying undirected graph is a tree and every in-degree is at most one, and an arborescence is $\emph{spanning}$ if the underlying undirected graph is a spanning tree. The unique vertex of an arborescence with in-degree $0$ is called its \emph{root}. The lack of a counterexample motivated Yokoi, one of the authors, to propose the following conjecture \cite[Open Problem: Rainbow Arborescence Problem]{de2019combinatorial}.

\begin{conj}[Rainbow Arborescence Conjecture]\label{conj:1}
Let $G$ be the disjoint union of $n-1$ spanning arborescences $A_1,\dots,A_{n-1}$ on the same set of $n$ vertices. Then $G$ has a spanning arborescence $A$ that intersects each $A_i$ in exactly one arc.
\end{conj}
We think of the input arborescences as having different colors, hence the name of the conjecture. In Conjecture~\ref{conj:1}, all the colors are used exactly once; that is, $A$ is a transversal rather than a partial transversal. 
Therefore, the rainbow arborescence conjecture is stronger than the matroid intersection conjecture specialized to spanning arborescences: the latter would only imply the existence of a branching of size $n-2$ whose arcs have different colors, where a \emph{branching} is a digraph whose connected components are arborescences. 

Recall that in the matroid conjectures mentioned above, the $-1$ term cannot be removed even in the case of bipartite matchings. 
There is an apparent structural similarity between the rainbow bipartite matching problem and rainbow arborescence problem, because the former can be viewed as the intersection of three partition matroids, while the latter is the intersection of two partition matroids and a graphic matroid. 
It is therefore particularly interesting that if Conjecture~\ref{conj:1} is true, then the arborescence problem has significantly more attractive properties, despite the apparent structural similarity.

\subsection{Our results}
\label{sec:our}

Let $G$ be a digraph on a vertex set $V$ of size $n\geq 2$, formed as the disjoint union of $\ell\geq 1$ spanning arborescences $A_1,\dots,A_{\ell}$, where each $i\in [\ell]$ represents a \emph{color}; parallel arcs of different colors may occur. For ease of discussion, we identify subgraphs with their arc sets and call a subgraph $A$ \emph{rainbow} if $|A\cap A_i|\le 1$ for every $i\in [\ell]$. To build some intuition for the rainbow arborescence conjecture, let us first present some further variants. 

\begin{conj}\label{conj:2}
Let $k\ge n-1$ and let $G$ be the disjoint union of $k$ spanning arborescences $A_1,\dots,A_{k}$ on the same set of $n$ vertices.
Then $G$ has a rainbow spanning arborescence $A$.
\end{conj}

\begin{conj}\label{conj:3}
Let $k \le n - 1$ and let $G$ be the disjoint union of $n-1$ spanning arborescences $A_1,\dots,A_{n-1}$ on the same set of $n$ vertices.
Then $G$ has a rainbow arborescence $A$ of size $k$.
\end{conj}

Both of these conjectures are equivalent to Conjecture~\ref{conj:1}, and may be easier to approach when $k\neq n-1$. Note that the statements are identical when $k=n-1$, while for $k\neq n-1$, Conjecture~\ref{conj:1} implies Conjectures~\ref{conj:2} and \ref{conj:3} by straightforward observations.

Our contribution has three main components. In \Cref{sec:hardness}, we study the computational complexity of finding a rainbow spanning arborescence with a prescribed root. More precisely, we show that deciding whether a rainbow spanning arborescence exists with a given root is NP-complete, even under the strong restriction that the input arborescences have exactly two possible roots (Theorem~\ref{thm:NP-hard}).

In \Cref{sec:relax}, we study relaxations of the conjecture in two directions. First, in \Cref{sec:higher}, we relax the number of input arborescences and prove that Conjecture~\ref{conj:2} holds when the number of arborescences is at least $2n-4$ (Theorem~\ref{thm:relax1}), improving on what follows from earlier results of Kotlar and Ziv \cite{kotlar2015rainbow}. As a corollary, we obtain an arborescence version of a result of Cameron and Wanless \cite{cameron2005covering} (Corollary~\ref{cor:relax1}). Then, in \Cref{sec:colorfularb}, we relax the size requirement of the output structure and show that Conjecture~\ref{conj:3} holds for $k\le\lfloor n/2\rfloor$ (Theorem~\ref{thm:relax2}).

Several special cases of Conjecture~\ref{conj:1} are settled in \Cref{sec:solved}. After developing some technical tools in \Cref{sec:lemmas} -- including a reduction that allows us to eliminate stars (Lemma~\ref{lem:star}) and a sufficient condition based on the set of multi-roots, i.e., vertices that serve as the root of at least two input arborescences (Lemma~\ref{lem:greedy}) -- we first verify the conjecture when each input arborescence is a path, or more generally, a path or a star (Theorem~\ref{thm:path} and Corollary~\ref{cor:path_star}). We then prove the conjecture for the case when there are at most two multi-roots (Theorem~\ref{thm:two_multiroots}), which in particular implies the result for all graphs on at most six vertices (Corollary~\ref{cor:six}). We further verified that the conjecture is indeed true up to eight vertices using a computer (Remark~\ref{rem:8}). 
Furthermore, we show that the conjecture holds when all input arborescences arise from orienting the edges of a common tree (Theorem~\ref{thm:tree}).

The result of the paper with the most involved proof appears in \Cref{sec:cycle}, which was also reported in our separate preprint~\cite{berczi2025rainbow}: Conjecture~\ref{conj:1} holds when the underlying undirected graph is a cycle (Theorem~\ref{thm:cycle}). Although this setting may appear simple at first glance, its proof is highly nontrivial and requires a careful analysis of certain blocking sets; for this reason, we defer the proof to \Cref{sec:proof}. Combined with the tree case, the theorem yields the conjecture when the underlying graph is a pseudotree (Corollary~\ref{cor:uni}). The cycle case further leads to an application to systems of distinct representatives for families of intervals on a cycle that might be of independent combinatorial interest: \emph{If $C$ is cycle of length $n$ and $I_1,\dots, I_n$ are arbitrary (not necessarily distinct) intervals of $C$, then there exists an interval $J$ of $C$ such that the family $I_1 \triangle J,\dots, I_n \triangle J$ has a system of distinct representatives} (Corollary~\ref{cor:representative}).

\subsection{Related work and motivation}
\label{sec:related}

The conjectures presented in the introduction already provide strong motivation for studying rainbow arborescences, particularly given the extensive work devoted to rainbow matchings in bipartite graphs (which includes partial transversals of Latin squares). Let us mention other lines of research that appear to be closely related to the problem we investigate.

\paragraph{Relaxations.} Kotlar and Ziv~\cite{kotlar2015rainbow} extended the work of Drisko~\cite{drisko1998proof} and Chappell~\cite{chappell1999matroid}, showing that for any $2k-1$ disjoint common independent sets of size $k$, a common independent rainbow set of size $k$ always exists. Recently, Berger and McGinnis~\cite{berger2025common} showed that if $A_1,\dots,A_{2k-1}$ are common independent sets in two matroids $M$ and $N$ with $|A_i|\ge\min\{i,k\}$ for each $i$, then a partial rainbow set of size $k$ exists that is independent in both matroids. For a comprehensive overview of results on Latin squares, we recommend the survey by Wanless~\cite{wanless2011transversals}.

\paragraph{Three matroid intersection.} The problems proposed in \cite{aharoni2013independent,aharoni2016rainbow}, as well as Conjecture~\ref{conj:1}, are special cases of finding a maximum-size common independent set of three matroids. While $3$-matroid intersection is oracle-hard (see e.g.~\cite{berczi2021complexity}), the question naturally arises whether reasonable lower bounds can be obtained using existing approximation methods. Unfortunately, neither the technique of Aharoni and Berger~\cite{aharoni2006intersection} nor the iterative refinement approach by Linhares, Olver, Swamy, and Zenklusen~\cite{linhares2020approximate} suffices to establish the strong lower bounds proposed in the conjectures.

\paragraph{Equitability and fair representations.}
A matroid $M=(E,\cB)$ is called \emph{$\ell$-equitable} if the ground set $E$ is the disjoint union of two bases and for any subpartition $A_1\cup\dots\cup A_\ell \subseteq E$, there exist disjoint bases $B_1,B_2 \in \cB$ such that $\lfloor |A_j|/2\rfloor\allowbreak \le |B_i\cap A_j|\le\lceil|A_j|/2\rceil$ for all $i\in[2]$ and $j\in[\ell]$; the case $\ell=1$ is usually referred to as \emph{equitability}. Fekete and Szabó~\cite{fekete2011equitable} conjectured that every matroid whose ground set is partitioned into two disjoint bases is equitable, a result recently proved by Akrami, Liu, Raj, and Végh~\cite{akrami2025matroids2}. They also showed that $\ell$-equitability fails for $\ell\ge 2$, but provided a mild relaxation instead: $2$-equitability holds whenever at least one of $|A_1|$ or $|A_2|$ is odd, and when both are even it can fail only by a single element on one part. 

Equitability is closely related to fair and almost fair representations of partitions~\cite{aharoni2017fairb}. For $\alpha>0$, a set $X$ \emph{$\alpha$-fairly} represents a set $S$ if $|X\cap S|\ge\lfloor \alpha|S|\rfloor$, and \emph{almost $\alpha$-fairly} if $|X\cap S|\ge\lfloor \alpha|S|\rfloor-1$. Aharoni, Berger, Kotlar, and Ziv~\cite{aharoni2017faira} conjectured that if the ground set $E$ can be partitioned into $k$ independent sets in each of two matroids, then every partition of $E$ admits a common independent set that almost $1/k$-fairly represents all its parts. The relaxation of $2$-equitability established by Akrami, Liu, Raj, and Végh~\cite{akrami2025matroids2} confirms this conjecture in the special case where the two matroids are dual and the partition has at most three parts. For two-partitions, a slightly weaker bound was previously obtained in~\cite{aharoni2017faira}.

A common theme underlying equitability, $\ell$-equitability, and fair or almost fair representations is finding a common basis of two matroids whose intersections with the parts of a given partition satisfy prescribed size constraints. Conjecture~\ref{conj:1} reflects this theme: it asserts that if a digraph on $n$ vertices is the union of $k$ arborescences $A_1,\dots,A_k$ and $\alpha_1,\dots,\alpha_k$ are nonnegative integers satisfying $\sum_{i=1}^k \alpha_i = n-1$, then there exists an arborescence $A$ with $|A\cap A_i| = \alpha_i$ for each $i\in[k]$. In particular, choosing each $\alpha_i\in\{\lfloor (n-1)/k\rfloor,\lceil (n-1)/k\rceil\}$ guarantees an arborescence that meets the input arborescences in nearly equal proportions.

\section{Complexity with fixed root}
\label{sec:hardness}

Observe that if all colors have the same root $r$, then a rainbow spanning arborescence rooted at $r$ exists and can be found by the following simple greedy approach: consider the colors in an arbitrary order, and in each step add an arc of the current color that leaves the vertex set of the already constructed rainbow arborescence. Furthermore, there is no rainbow spanning arborescence rooted at a different vertex, since $r$ has no incoming arc in $G$. In the general setting, a natural approach to find a rainbow spanning arborescence would be to fix a root, start with an empty arc set, and augment it by adding an arc of each color one-by-one. It, however, turns out that finding a rainbow arborescence rooted at a fixed vertex is hard. More precisely, let \textsc{Rooted-RA} denote the problem where the input is a directed graph $G$ that is the disjoint union of $n-1$ spanning arborescences $A_1, \dots, A_{n-1}$ (with arbitrary roots) along with a target root vertex $r$, and the goal is to decide if $G$ contains a rainbow spanning arborescence rooted at $r$.

\begin{thm}\label{thm:NP-hard}
    \textsc{Rooted-RA} is NP-complete even if exactly two vertices are roots of the input arborescences.  
\end{thm}
\begin{proof}
The problem is clearly in NP\@.
We prove hardness by reduction from the 3-dimensional matching problem, denoted by \textsc{3DM}\@.
In \textsc{3DM}, we are given a tripartite $3$-uniform hypergraph $\cH=(X,Y,Z;\mathcal{E})$ where $|X|=|Y|=|Z|=p$, $|\cE| = q$, and $\cE\subseteq X\times Y\times Z$, and the goal is to decide whether $\cH$ has a perfect matching.
This problem is known to be NP-complete~\cite[(SP2)]{garey1979computers}.

Let $\cH=(X,Y,Z;\cE)$ be an instance of \textsc{3DM}\@.
We construct an instance of \textsc{Rooted-RA} as follows.
Let $D_{st}$ denote the digraph obtained by taking two vertices $s$ and $t$, and adding a directed $s$--$t$ path $P_H=(e_{H,x},e_{H,y},e_{H,z},f_H)$ of length $4$ for each hyperedge $H=(x,y,z)$, where all the paths are openly disjoint; thus, $D_{st}$ consists of $3q + 2$ vertices and $4q$ arcs.
We take $p$ copies of $D_{st}$ and apply the notational convention that in the $j$th copy, we add $j$ as an upper index to the notation.
We connect these $p$ digraphs $D_{st}^j$ by identifying $t^j$ with $s^{j+1}$ for $j\in[p-1]$.
We also add an extra vertex $t$.

We assign color $c_v$ to all the arcs of the form $e_{H,v}^j$ for each $v\in X\cup Y\cup Z$.
Furthermore, we assign a unique color $c_{H,j}$ to every arc of the form $f^j_H$.
Note that at this point, every color class forms a branching. We extend these branchings to spanning arborescences as follows.
For each $v\in X\cup Y\cup Z$, we add an arc of color $c_v$ from $t_p$ to every other vertex with no incoming arc of color $c_v$.
That is, the root of the arborescence of color $c_v$ is $t_p$.
For each $H\in\cE$ and $j\in[p]$, we add an arc of color $c_{H,j}$ from $t$ to every other vertex with no incoming arc of color $c_{H,j}$.
That is, the root of the arborescence of color $c_{H,j}$ is $t$.
Finally, we add $2pq-2p+1$ stars rooted at $t_p$ of distinct colors.
Let $G = (V, E)$ be the digraph thus obtained.
It is easy to check from definition that $|V|=3pq+p+2$ and $E$ is the union of $3pq+p+1$ spanning arborescences.
We claim that $G$ has a rainbow spanning arborescence rooted at $s_1$ if and only if $\cH$ has a perfect matching. 

Assume first that $G$ has a rainbow spanning arborescence rooted at $s_1$ and $\cH$ has no perfect matching.
By construction, the arborescence must contain a rainbow $s_1$--$t_p$ path.
This path is the union of $s_j$--$t_j$ paths for $j\in[p]$, each corresponding to some hyperedge $H_j\in\cE$.
Since the $s_1$--$t_p$ path is rainbow, the hyperedges $H_1,\dots,H_p$ must be disjoint and hence is a perfect matching in $\cH$, a contradiction.

Now suppose that $\cH$ has a perfect matching, consisting of hyperedges $H_1,\dots,H_p\in\cE$.
The union of the paths $P^j_{H_j}$ is a rainbow $s_1$--$t_p$ path.
Using one of the stars rooted at $t_p$, we can extend this path to a rainbow $s_1$--$t$ path.
Finally, using the remaining stars and the colors $c_{H,j}$ rooted at $t$, we can finish the construction to obtain a rainbow spanning arborescence rooted at $s_1$. 
\end{proof}

\section{Relaxations}
\label{sec:relax}

In this section, we consider relaxations of Conjecture~\ref{conj:1} in two directions. First, we relax the number of input arborescences (Section~\ref{sec:higher}), and then we relax the size of the rainbow arboresence to be found (Section~\ref{sec:colorfularb}). 

\subsection{Relaxing the number of input arborescences}
\label{sec:higher}

The result of Kotlar and Ziv~\cite{kotlar2015rainbow}, when specialized to arborescences, implies that if $G$ is the disjoint union of $2n-3$ spanning arborescences, then there exists a rainbow spanning arborescence.
We improve this bound to $2n-4$ using a simpler proof (assuming Theorem~\ref{thm:two_multiroots}).

\begin{thm}\label{thm:relax1}
    Conjecture~\ref{conj:2} is true when $k \ge 2n - 3$. Furthermore, if $n \ge 3$, Conjecture~\ref{conj:2} is true when $k \ge 2n - 4$.
\end{thm}

\begin{proof}
    We prove the first claim by induction on $n$.
    The base case is $n = 2$, and it is trivial.
    Suppose that $n \ge 3$.
    Let $\tilde{A}$ be an inclusion-wise maximal rainbow arborescence in $G$ and $\tilde{V}$ be its vertex set; that is, $\tilde{A}$ is a rainbow arborescence on $\tilde{V}$ such that for any color $i \in [k]$ with $\tilde{A} \cap A_i = \emptyset$ and any arc $(u, v) \in A_i$, we have $u \not\in \tilde{V}$ or $v \in \tilde{V}$.
    If $\tilde{V} = V$, then we have obtained a desired arborescence.
    Otherwise, for every color $i \in [k]$ with $\tilde{A} \cap A_i = \emptyset$, the root of $A_i$ must be in $V' \coloneqq V \setminus \tilde{V}$, and the restriction $A'_i$ of $A_i$ to $V \setminus \tilde{V}$ is also an arborescence.
    The number of such colors is
    \[k' \coloneqq k - (|\tilde{V}| - 1) \ge 2n - |\tilde{V}| - 2.\]
    As $\tilde{V} \neq \emptyset$, we obtain $k' \ge 2(n - |\tilde{V}|) - 1 > 2|V'| - 3$.
    Thus, by induction hypothesis if $|V'| \ge 2$ and trivially if $|V'| = 1$, there exists a rainbow arborescence $A'$ on $V'$ only using these $k'$ colors.
    Also, since $k' \ge n - 1$ as $|\tilde{V}| \le n - 1$, starting with an arborescence $A'$ on $V'$, we can obtain a rainbow arborescence $A$ on $V$ by adding outgoing arcs of unused ones among the $k'$ colors one-by-one (as they have roots in $V'$).

    Finally, we verify the second claim.
    Suppose that $n \ge 3$ and $k = 2n - 4$ (note that $k \ge n - 1$).
    If we fail the induction proof above, then we have $k' < n - 1$, which implies $|\tilde{V}| = n - 1$ and $k' = n - 2$.
    In this case, let $r$ be the unique vertex in $V'$, and then we can construct a rainbow arborescence $A'$ rooted at $r$ with $|A'| = k' = n - 2$ by adding outgoing arcs of the $k'$ colors one-by-one.
    If we can add one more arc of an unused color, then we obtain a desired arborescence.
    Otherwise, all the $k - k' = n - 2$ unused colors have the same root, which is a unique vertex not spanned by $A'$.
    Thus, $G$ has at most two multi-roots, and Theorem~\ref{thm:two_multiroots} completes the proof.
\end{proof}

As a corollary, we get a counterpart of the result of Cameron and Wanless~\cite{cameron2005covering}, but for arborescences. 

\begin{cor}\label{cor:relax1}
    Let $n\geq 3$ and $G$ be the disjoint union of $n - 2$ spanning arborescences $A_1, A_2, \dots, A_{n-2}$ on $V$.
    Then, $G$ has a spanning arborescence using at most two arcs from each $A_i$.
\end{cor}
\begin{proof}
    For each arborescence $A_i$, make its copy $A'_i$.
    Then, by Theorem~\ref{thm:relax1} (with $k = 2n - 4$), there exists a spanning arborescence $A'$ on $V$ using at most one arc from each $A_i$ and $A'_i$.
    Thus, one can obtain a desired spanning arborescence $A$ from $A'$ by replacing each copy arc from $A'_i$ with its original in $A_i$.
\end{proof}

\subsection{Relaxing the size of output arborescence}
\label{sec:colorfularb}

The results of Aharoni, Kotlar, and Ziv~\cite{aharoni2013independent,aharoni2016rainbow} and of Kotlar and Ziv~\cite{kotlar2015rainbow} both imply that if $G$ is the disjoint union of $n-1$ spanning arborescences, then it has a rainbow branching of size $\lfloor n/2\rfloor$ -- possibly consisting of several arborescences. We prove a strengthening of this observation by showing that the branching in fact can be chosen to be a single arborescence.

\begin{thm}\label{thm:relax2}
    Conjecture~\ref{conj:3} is true when $k \le \lfloor n/2 \rfloor$.
\end{thm}

\begin{proof}
    As in the proof of Theorem~\ref{thm:relax1}, let $\tilde{A}$ be an inclusion-wise maximal rainbow arborescence in $G$ and $\tilde{V}$ be its vertex set.
    If $|\tilde{A}| \ge \lfloor n/2 \rfloor$, then we have obtained a desired arborescence.
    We show that otherwise we can augment $(\tilde{A}, \tilde{V})$ by exchanging several arcs.

    Let $\tilde{r}$ be the root of $\tilde{A}$, and $J$ be the set of unused colors.
    Since $\tilde{A}$ is inclusion-wise maximal, for every color $i \in J$, the root of $A_i$ is not in $\tilde{V}$, and the restriction of $A_i$ to $V\setminus \tilde{V}$ is also an arborescence.
    Thus, each $A_i$ $(i \in J)$ includes a path $P_i$ satisfying the following three conditions:
    \begin{itemize}\itemsep0em
        \item the root is not in $\tilde{V}$,
        \item all other vertices are in $\tilde{V}$, and
        \item the leaf is $\tilde{r}$.
    \end{itemize}

\begin{cl}\label{cl:rainbow_path}
    The disjoint union $\bigcup_{i \in J} P_i$ contains a rainbow path $P$ satisfying the same three conditions.
\end{cl}
\begin{proof}
    
    It suffices to show that $\bigcup_{i \in J} P_i$ has a rainbow in-arborescence $A'$ on $V' \not\subseteq \tilde{V}$ rooted at $\tilde{r}$, where an \emph{in-arborescence} is a digraph which becomes an arborescence by flipping the direction of all arcs (in other words, the underlying graph is a tree on $V'$ and the out-degree in $V' \setminus \{\tilde{r}\}$ is one).
    It is clear that such an $A'$ must contain a desired rainbow path $P$.
    
    We prove this constructively.
    Initially, let $A' = \emptyset$ and $V' = \{\tilde{r}\}$.
    For each color $i \in J$ (in any order), we do the following update.
    If $V' \not\subseteq \tilde{V}$ at some point, then we have obtained a desired in-arborescence $A'$.
    Otherwise, $P_i$ has an arc $e = (u, v)$ such that $u \notin V'$ and $v \in V'$ (as $P_i$ starts outside of $\tilde{V} \supseteq V'$ and ends at $\tilde{r} \in V'$); we add $e$ to $A'$ and $u$ to $V'$.
    By the assumptions, we have
    \[|\tilde{V}| = |\tilde{A}| + 1 \le \lfloor n/2 \rfloor \le n - 1 - (\lfloor n/2 \rfloor - 1) \le n - 1 - |\tilde{A}| = |J|.\]
    This implies $|V'| = |J| + 1 > |\tilde{V}|$ at the end, and hence $V' \not\subseteq \tilde{V}$.
\end{proof}
    Let $P$ be a path obtained by Claim~\ref{cl:rainbow_path}, and let $Q$ be the set of arcs $e \in \tilde{A}$ such that the head of $e$ is in $P$.
    Then, $|Q| = |P| - 1$ as $\tilde{r}$ is the root of the arborescence $\tilde{A}$, and we can indeed augment $\tilde{A}$ to $(\tilde{A} \cup P) \setminus Q$.
    This completes the proof of the theorem.
\end{proof}

\section{Special cases}
\label{sec:solved}

As a step toward verifying Conjecture~\ref{conj:1}, we show that the statement is true under several additional assumptions: when each $A_i$ is a path or a star (Section~\ref{sec:path}), when there are at most two vertices that are roots of multiple input arborescences (Section~\ref{sec:tworoots}), and when the underlying graph is a tree (Section~\ref{sec:tree}). Our main result is the verification of the conjecture in the seemingly simple setting where the underlying graph is a cycle (Section~\ref{sec:cycle}). However, this apparent simplicity is misleading, as the proof is highly nontrivial; for this reason, it is deferred to Section~\ref{sec:proof}.

\subsection{Preparations}
\label{sec:lemmas}

For a digraph $G$ and vertex $v$, we write $\delta_G^{\mathrm{in}}(v)$ and $\delta_G^{\mathrm{out}}(v)$ for the sets of arcs \emph{entering} and \emph{leaving} $v$, respectively. For an arborescence and a vertex that is not the root, the tail of its unique incoming arc is its \emph{parent}, and a vertex with no outgoing arc is a \emph{leaf}. Special cases include \emph{paths}, which have exactly one leaf, and \emph{stars}, in which every non-root vertex is a leaf.

For the proofs, we need some technical lemmas. We first observe that we can get rid of stars among the arborescences $A_i$. However, since the reduction may change the structure of other arborescences, we should be careful when we use this lemma in an argument based on induction.

\begin{lem}\label{lem:star}
    Let $G$ be the disjoint union of $n - 1$ spanning arborescences $A_1, A_2, \dots, A_{n-1}$ on $V$, and suppose that $A_{n-1}$ is a star rooted at $r$.
    For each $i \in [n-2]$, we define an arborescence $A'_i$ on $V - r$ as follows.
    \begin{itemize}\itemsep0em
        \item If the root of $A_i$ is $r$, then choose any vertex $r' \in V - r$ with $(r, r') \in \delta^{\mathrm{out}}_{A_i}(r)$ and let $A'_i \coloneqq \{ (r', v) \colon (r, v) \in \delta^{\mathrm{out}}_{A_i}(r),\ v \neq r' \} \cup (A_i \setminus \delta^{\mathrm{out}}_{A_i}(r))$.
        \item Otherwise, let $r'$ be the parent of $r$ on $A_i$ and $A'_i \coloneqq \{ (r', v) \colon (r, v) \in \delta^{\mathrm{out}}_{A_i}(r) \} \cup (A_i \setminus (\delta^{\mathrm{out}}_{A_i}(r) \cup \{(r', r)\}))$.
    \end{itemize}
    In particular, if $A_i$ is a star, then $A'_i$ is also a star.
    If there exists a rainbow arborescence on $V - r$ with respect to $(A'_1, A'_2, \dots, A'_{n-2})$, then this is true on $V$ with respect to $(A_1, A_2, \dots, A_{n-1})$.
\end{lem}

\begin{proof}
Let $A'$ be a rainbow arborescence on $V - r$ with respect to $(A'_1, A'_2, \dots, A'_{n-2})$, and define a subgraph $\hat{A}$ as follows.
For each $i \in [n-2]$, let $\{e'_i\} = A' \cap A'_i$.
If $e'_i \in A_i$ as it is, then $\hat{A}$ contains the arc $e'_i$.
Otherwise, by definition, $e'_i = (r', v)$ for some $v \in V - r$, where $r'$ is the root of $A'_i$, and it corresponds to two arcs in $A_i$ as follows.
\begin{itemize}\itemsep0em
    \item If the root of $A_i$ is $r$, then $e'_i$ corresponds to $(r, r'), (r, v) \in A_i$.
    \item Otherwise, $e'_i$ corresponds to $(r', r), (r, v) \in A_i$.
\end{itemize}
In either case, $\hat{A}$ contains both of the two corresponding arcs.

Suppose that $|\hat{A}| = n - 2$.
In this case, $\hat{A} = A'$, and we obtain a desired arborescence on $V$ just by adding an arc in the star $A_{n-1}$ from $r$ to the root of $A'$.

Otherwise, we have $|\hat{A}| \ge n - 1$.
Then, $|\hat{A} \cap A_i| = 2$ for at least one index $i \in [n-2]$, and $\hat{A} \cap A_i$ contains an arc $(r, v)$ such that $e'_i = (r', v)$.
Let $A$ be the subgraph of $G$ obtained from $\hat{A}$ by removing the other arc for each of such indices $i$ ($(r, r')$ or $(r', r)$ depending on the root of $A_i$).
Then, $A$ is disjoint only from $A_{n-1}$ and consists of two disjoint arborescences, one of whose root is $r$.
Thus, by adding an arc in the star $A_{n-1}$ from $r$ to the other root, we obtain a rainbow spanning arborescence.
\end{proof}

For each vertex $v \in V$, let $\rho(v) \coloneqq |\{ i \in [n-1] \colon \text{the root of $A_i$ is $v$}\}|$.
A vertex $v$ is called a \emph{multi-root} if $\rho(v) \ge 2$, and a \emph{non-root} if $\rho(v) = 0$.
We then observe the following sufficient condition for the existence of a rainbow spanning arborescence.

\begin{lem}\label{lem:greedy}
    Let $G$ be the disjoint union of $n - 1$ spanning arborescences $A_1, A_2, \dots, A_{n-1}$ on $V$.
    If $G$ has a rainbow arborescence $\tilde{A}$ on $\tilde{V} \subseteq V$ such that $\tilde{V}$ contains all the multi-roots, then $G$ has a rainbow spanning arborescence.
\end{lem}

\begin{proof}
    We show that we can augment $(\tilde{A}, \tilde{V})$ to $(\tilde{A}', \tilde{V}')$ with $\tilde{V}' = \tilde{V} + v$ for some $v \in V \setminus \tilde{V}$ unless $\tilde{V} = V$.
    Suppose that there exists a color $i \in [n-1]$ with $|\tilde{A} \cap A_i| = 0$ and $r_i \in \tilde{V}$, where $r_i$ is the root of $A_i$.
    Then, $A_i$ has an arc $e = (u, v)$ with $u \in \tilde{V}$ and $v \not\in \tilde{V}$, and hence we just define $\tilde{A}' \coloneqq \tilde{A} + e$ and $\tilde{V} \coloneqq \tilde{V} + v$.
    Thus, we may assume that $|\tilde{A} \cap A_i| = 1$ for any color $i$ with $r_i \in \tilde{V}$.
    This implies $\sum_{v \in \tilde{V}} \rho(v) \le |\tilde{A}| = |\tilde{V}| - 1$.
    Since $\sum_{v \in V}\rho(v) = |V| - 1$ and $\rho(v) \le 1$ for every $v \in V \setminus \tilde{V}$, we also have $\sum_{v \in \tilde{V}} \rho(v) \ge |\tilde{V}| - 1$.
    Thus, the equality holds, and $|\tilde{A} \cap A_i| = 1$ if and only if $r_i \in \tilde{V}$.

    Let $\tilde{r}$ be the root of $\tilde{A}$.
    Pick an unused color $i$, and let $e = (v, \tilde{r})$ be the unique arc in $\delta_{A_i}^{\mathrm{in}}(\tilde{r})$.
    If $v \not\in \tilde{V}$, then we define $\tilde{A}' \coloneqq \tilde{A} + e$ and $\tilde{V}' \coloneqq \tilde{V} + v$.
    Otherwise, let $f = (u, v)$ be the unique arc in $\delta_{\tilde{A}}^{\mathrm{in}}(v)$ and we first update $\tilde{A}$ to $\tilde{A} + e - f$.
    After this update, $\tilde{A}$ is still a rainbow arborescence on $\tilde{V}$, whose root is $v$, and satisfies $|\tilde{A} \cap A_j| = 0$ and $r_j \in \tilde{V}$ for the color $j$ with $f \in A_j$.
    Thus, $A_j$ has an arc $e' = (u', v')$ with $u' \in \tilde{V}$ and $v' \not\in \tilde{V}$, and we can define $\tilde{A}' \coloneqq \tilde{A} + e'$ and $\tilde{V}' \coloneqq \tilde{V} + v'$.
\end{proof}

\subsection{Each arborescence being a path or a star}
\label{sec:path}

We first verify the conjecture when each arborescence is a path.

\begin{thm}\label{thm:path}
    Conjecture~\ref{conj:1} is true when each $A_i$ is a path.
\end{thm}

\begin{proof}
Since each $A_i$ is a path, it has the unique leaf.
We first claim the existence of a pair of a subgraph $S$ and a vertex $r\in V$ satisfying the following three conditions:
\begin{enumerate}[label = (\arabic*)] \itemsep0em
\item $|S\cap A_i|=1$ for each $i\in [n-1]$, \label{it:1}
\item $|\delta_S^{\mathrm{in}}(v)|=1$ for each $v\in V\setminus \{r\}$ and $|\delta_S^{\mathrm{in}}(r)|=0$, and \label{it:2} 
\item for each $i\in [n-1]$, the unique arc in $S\cap A_i$ lies between $r$ and the leaf on the path $A_i$. \label{it:3}
\end{enumerate}
Indeed, such $(S, r)$ can be found by the following procedure. Initially, let $S$ be empty. Then, for each $i=1,2,\dots, n-1$, do the following.
Let $v_i$ be the vertex on $A_i$ closest to the leaf among those that have no incoming arcs in $S$ and add to $S$ the arc of $A_i$ entering $v_i$.
Then, the size of the resulting $S$ is $n-1$. Let $r$ be the unique vertex that has no incoming arc in $S$. We see that these $S$ and $r$ satisfy conditions \ref{it:1}--\ref{it:3}.

Among the pairs $(S, r)$ satisfying conditions \ref{it:1}--\ref{it:3}, take the one minimizing $\sum_{i\in [n-1]}|A_i(r,S)|$, where $|A_i(r,S)|$ is the length of the path on $A_i$ from $r$ to the head of the unique arc in $S\cap A_i$.
We show that $S$ is an arborescence, which completes the proof by condition~\ref{it:1}.

Suppose, to the contrary, that $S$ is not so.
Since we have conditions \ref{it:1} and \ref{it:2}, $S$ is not connected, and at least one connected component must contain a cycle.
That is, there is a vertex set $\{u^1, u^2, \dots, u^k\}$ such that $(u^j,u^{j+1})\in S$ for $j\in [k]$, where we let $u^{k+1}=u^1$. Note that $r$ is not contained in this cycle as we have $|\delta_S^{\mathrm{in}}(r)|=0$.
Let $S'$ be a subgraph obtained from $S$ by replacing each arc $(u^j,u^{j+1})$ in the cycle with the arc of the same color that enters $u^j$.
Such an arc does exist because $u^j\neq r$ and we have condition \ref{it:3}.
This $S'$ satisfies conditions \ref{it:1}--\ref{it:3} with the same $r$ while $\sum_{i\in [n-1]}|A_i(r,S')|<\sum_{i\in [n-1]}|A_i(r,S)|$, which contradicts the choice of $S$.
\end{proof}

By combining Theorem~\ref{thm:path} with Lemma~\ref{lem:star}, we get the following corollary.

\begin{cor}\label{cor:path_star}
    Conjecture~\ref{conj:1} is true when each $A_i$ is either a path or a star.
\end{cor}

\begin{proof}
We prove this by induction on $n$.
If all $A_i$ are paths, just apply Theorem~\ref{thm:path}.
Otherwise, without loss of generality, we may assume that $A_{n-1}$ is a star.
Then, by Lemma~\ref{lem:star}, we can reduce the instance by removing the star and its root.
By the induction hypothesis, it suffices to show that the reduction given in Lemma~\ref{lem:star} preserves the condition that an arborescence is a path.

Suppose that $A_i$ $(i \in [n-2])$ is a path.
If the root of $A_i$ is $r$, then $A'_i = A_i \setminus \{(r, v)\}$, where $v$ is the next vertex of $r$ on the path $A_i$; thus, this is indeed a path.
Otherwise, $A'_i = A_i \setminus \{(r', r)\}$ if $r$ is the leaf, and $A'_i = (A_i \setminus \{(r', r), (r, v)\}) \cup \{(r', v)\}$ otherwise, where $r'$ is the parent of $r$ and $v$ is the next vertex of $r$ on the path $A_i$; thus, this is also a path.
\end{proof}

\subsection{At most two multi-roots}
\label{sec:tworoots}

Our next result shows that the conjecture also holds if the number of multi-roots is small.
Recall that an in-arborescence is a digraph which becomes an arborescence by flipping the direction of all arcs.

\begin{thm}\label{thm:two_multiroots}
    Conjecture~\ref{conj:1} is true when there exist at most two multi-roots.
\end{thm}

\begin{proof}
When there exists at most one multi-root, the claim follows from Lemma~\ref{lem:greedy} by taking $\tilde{A} = \emptyset$ and $\tilde{V} = \{r\}$ for a vertex $r \in V$ with $\rho(r) = \max_{v \in V} \rho(v)$.
The remaining case is when there exist exactly two multi-roots, say $x_1$ and $x_2$.
By Lemma~\ref{lem:greedy} again, it suffices to show that $G$ has a rainbow arborescence $\tilde{A}$ on some $\tilde{V}$ with $\{x_1, x_2\} \subseteq \tilde{V}$.

In order to obtain such $(\tilde{A}, \tilde{V})$, we show that $G$ has a pair of arc-disjoint in-arborescences $\tilde{B}_1$ and $\tilde{B}_2$ on $\tilde{V}_1$ and $\tilde{V}_2$, respectively, satisfying the following three conditions:
\begin{enumerate}[label = (\arabic*)] \itemsep0em
    \item $\tilde{B}_j$ is rooted at $x_j$ for each $j \in [2]$,\label{it:m1}
    \item $\tilde{B}_1 \cup \tilde{B}_2$ is rainbow, and \label{it:m2}
    \item $\tilde{V}_1 \cap \tilde{V}_2 \neq \emptyset$. \label{it:m3}
\end{enumerate}
It is easy to see that $\tilde{B}_1 \cup \tilde{B}_2$ includes a desired rainbow arborescence $\tilde{A}$ as it contains arc-disjoint paths from each $\tilde{r} \in \tilde{V}_1 \cap \tilde{V}_2$ to $x_1$ and $x_2$.

Observe that if $|\tilde{B}_1 \cup \tilde{B}_2| = n - 1$ in addition to conditions \ref{it:m1}--\ref{it:m2}, then condition \ref{it:m3} is automatically satisfied by the pigeonhole principle (as $|\tilde{V}_1| + |\tilde{V}_2| = |\tilde{B}_1| + |\tilde{B}_2| + 2 = n + 1 > |V|$).
Also, $\tilde{B}_1 = \tilde{B}_2 = \emptyset$, $\tilde{V}_1 = \{x_1\}$, and $\tilde{V}_2 = \{x_2\}$ satisfy conditions \ref{it:m1}--\ref{it:m2}.
Thus, it suffices to show that we can augment at least one of the pairs $(\tilde{B}_1, \tilde{V}_1)$ and $(\tilde{B}_2, \tilde{V}_2)$ satisfying conditions \ref{it:m1}--\ref{it:m2} as long as $\tilde{V}_1 \cap \tilde{V}_2 = \emptyset$.

Suppose that $\tilde{V}_1 \cap \tilde{V}_2 = \emptyset$.
Then, there exists a color $i \in [n-1]$ such that $(\tilde{B}_1 \cup \tilde{B}_2) \cap A_i = \emptyset$.
As $\tilde{V}_1 \cap \tilde{V}_2 = \emptyset$, at least one $\tilde{V}_j$ does not contain the root of $A_i$.
Thus, $A_i$ contains some arc $e = (u, v) \in A_i$ with $u \not\in \tilde{V}_{j}$ and $v \in \tilde{V}_j$, and we can indeed augment $(\tilde{B}_j, \tilde{V}_j)$ to $(\tilde{B}_j + e, \tilde{V}_j + u)$.
\end{proof}

As a corollary, we get that the conjecture is true for graphs on a small number of vertices.

\begin{cor}\label{cor:six}
    Conjecture~\ref{conj:1} is true when $n \le 6$.
\end{cor}

\begin{proof}
    When $n \le 6$, there are at most five colors.
    Thus, there are at most two multi-roots, and Theorem~\ref{thm:two_multiroots} completes the proof.
\end{proof}

\begin{rem}\label{rem:8}
We verified using a computer that Conjecture~\ref{conj:1} is indeed true when $n \le 8$.
Note that the number of possible instances is naively estimated as $\left(n^{n-1}\right)^{n-1} = n^{(n-1)^2}$, since the number of spanning trees on $n$ labeled vertices is $n^{n-2}$ by Cayley's theorem~\cite{arthur1888theorem,moon1967various} and then that of spanning arborescences is $n^{n-1}$ (including $n$ possible choices of the root).
Thus, even when $n = 7$, the literally brute-force search checks $7^{36} \approx 2.65 \times 10^{30}$ instances, which is impractical.
With the aid of theory, in particular using Theorem~\ref{thm:two_multiroots} and Lemma~\ref{lem:greedy}, we can employ several pruning strategies, which enable us to complete up to $n = 8$.

The high level idea is the following.
Let $V = [n]$.
First, we fix the roots of all colors; let $r_i$ denote the root of color $i \in [n-1]$.
By symmetry and Theorem~\ref{thm:two_multiroots}, it suffices to consider the case $(r_1, r_2, \dots, r_6) = (1, 1, 2, 2, 3, 3)$ when $n = 7$ and the cases $(r_1, r_2, \dots, r_7) = (1, 1, 1, 2, 2, 3, 3), (1, 1, 2, 2, 3, 3, 4)$ when $n = 8$.
Then, starting with empty branchings $A'_i$ $(i \in [n-1])$, we grow them to spanning arborescences with their roots $r_i$ at the end in a DFS manner; that is, in each step, we choose a color $i$ and a vertex $v \neq r_i$ that has no incoming arc of color $i$, and add a possible arc to $A'_i$ that enters $v$.
During this process, we maintain the family of possible rainbow branchings that consist of arcs already added, and if there is a rainbow (not necessarily spanning) arborescence containing all the multi-roots (the three vertices $1, 2, 3$ when $n = 7, 8$), then we prune the current branch since we can obtain a rainbow spanning arborescence by Lemma~\ref{lem:greedy} no matter what arcs will be added hereafter.
In addition, if the current branch is completely symmetric with another branch that has already been completed (i.e., there exists a pair of a permutation of the vertex indices and a permutation of the color indices making the two branches coincide), then we prune it.
By incorporating a heuristic strategy for deciding which pair of a color and a vertex should be chosen in the next DFS step to make pruning occur as early as possible, the total computation is completed within a few minutes for $n = 8$.
The source code is available in \cite{yycode}.
\end{rem}

\subsection{Underlying graph is a tree}
\label{sec:tree}

In this section, we show that the conjecture is true if the underlying graph is a tree, i.e., each arborescence is obtained by orienting the edges of the same tree.

\begin{thm}\label{thm:tree}
    Conjecture~\ref{conj:1} is true when the underlying graph of $G$ is a tree.
\end{thm}

\begin{proof}
We prove the statement by induction on $n$.
    The base case is $n = 2$, and it is trivial.
    Suppose that $n \ge 3$.
    Let $v \in V$ be a leaf of $A_{n-1}$ and $e = (u, v) \in A_{n-1}$ be its incoming arc.
    For each $i \in [n-2]$, let $A'_i$ be the arborescence on $V - v$ obtained from $A_i$ by removing the unique arc incident to $v$.
    By induction hypothesis, there exists an arborescence $A'$ on $V - v$ such that $|A' \cap A'_i| = 1$ for all $i \in [n-2]$.
    Thus, we obtain a desired rainbow arborescence $A = A' + e$ on $V$.
\end{proof}

\subsection{Underlying graph is a cycle}
\label{sec:cycle}

In terms of the structure of the graph, one of the simplest open cases of the rainbow arborescence conjecture is when the underlying undirected graph is a cycle (with possible parallel edges, which are inevitable as there are $(n-1)^2$ arcs).
As the main result of this paper, we show that the conjecture is true in this case. Interestingly, although the structure seems simple, the conjecture turned out to be much more difficult to prove than in cases settled in the previous subsections; we give the proof in Section~\ref{sec:proof}. 

\begin{thm}\label{thm:cycle}
    Conjecture~\ref{conj:1} is true when the underlying graph is a cycle. 
\end{thm}

Let us note that, in contrast to the general case (Theorem~\ref{thm:NP-hard}), deciding whether there is a rainbow arborescence with a given root $r$ is polynomial-time solvable on a cycle, since there are $n$ possible spanning arborescences rooted at $r$, and we can check for all of them whether they are rainbow or not via bipartite matching. This, together with Theorem~\ref{thm:cycle}, implies that if the underlying undirected graph is a cycle, then a rainbow arborescence can be found in polynomial time by testing each vertex as the root.

By combining Theorem \ref{thm:cycle} with the argument for the case where the underlying graph is a tree (i.e., the proof of Theorem~\ref{thm:tree}), one can easily see that the conjecture is also true when the underlying graph is a \emph{pseudotree} (a connected graph having at most one cycle).

\begin{cor}\label{cor:uni}
Conjecture~\ref{conj:1} is true when the underlying graph is a pseudotree.
\end{cor}
\begin{proof}
     Suppose that $v$ is a leaf of the underlying graph. If $v$ has an incoming arc of some color $c$, then one can reduce the instance by removing $v$ and $c$ as in the proof of Theorem~\ref{thm:tree}; otherwise, $v$ is the root of all input arborescences and then one can construct a rainbow spanning arborescence greedily, as described at the beginning of Section \ref{sec:hardness}.
\end{proof}

In addition, we also show that Theorem \ref{thm:cycle} implies an interesting new result on systems of distinct representatives of a family of intervals on the cycle. 
Let $C=(V,E)$ be a cycle of length $n$. An edge set $I\subseteq E$ is called an \emph{interval} of $C$ if it is connected, that is, $I$ is a path or $I \in\{\emptyset,E\}$.

Given a family $E_1,\dots,E_k$ of (not necessarily distinct) subsets of $E$, a \emph{system of distinct representatives} for $E_1,\dots,E_k$ is an edge set $S \subseteq E$ of size $k$ and a bijection $\sigma\colon S \to [k]$ such that $e\in E_{\sigma(e)}$ for every $e \in S$.
For two sets $X$ and $Y$, their symmetric difference is denoted by $X \triangle Y = (X \setminus Y) \cup (Y \setminus X)$. 

\begin{cor}\label{cor:representative}
    Let $C=(V,E)$ be a cycle of length $n$, and let $I_1,\dots, I_n$ be arbitrary (not necessarily distinct) intervals of $C$. Then there exists an interval $J$ of $C$ such that the family $I_1 \triangle J,\dots, I_n \triangle J$ has a system of distinct representatives.
\end{cor}
\begin{proof}
    We have to show that there exist an interval $J$ of $C$ and a bijection $\sigma\colon E \to [n]$ such that $e\in I_{\sigma(e)}\triangle  J$ for every $e \in E$. In order to reduce this to Theorem \ref{thm:cycle}, we consider the equivalent directed version where $C=(V,E)$ is a directed cycle and $I_1,\dots,I_n$ are directed intervals. We define the bidirected cycle $G=(V,E \cup F)$ where $F$ consists of the reversed edges of $E$, and we call $E$ the set of \emph{clockwise edges} and $F$ the set of \emph{anticlockwise edges}. For each interval $I_i$ ($i\in [n-1]$), we define a spanning arborescence $A_i$ of $G$ the following way. 
    \begin{itemize}\itemsep0em
        \item If $I_i=\emptyset$, then $A_i$ is an arbitrary anticlockwise path of length $n-1$.
        \item If $I_i=E$, then $A_i$ is an arbitrary clockwise path of length $n-1$.
        \item If $I_i$ is a path, then $A_i$ is the spanning arborescence of $G$ whose clockwise path is $I_i$. 
    \end{itemize}
    Note that in each case the clockwise path of $A_i$ is included in $I_i$ and the reverse of the anticlockwise path of $A_i$ is included in $E \setminus I_i$.
    By Theorem \ref{thm:cycle}, we can pick arcs $a_i \in A_i$ ($i \in [n-1]$) such that $A=\{a_1,\dots,a_{n-1}\}$ is a spanning arborescence of $G$. Let $P \subseteq E$ be the clockwise path of $A$ and let $P'\subseteq E$ be the reverse of the anticlockwise path of $A$.

    Let $e^*=E\setminus (P\cup P')$. If $e^* \in I_n$, then let $J=P'$. If $e^* \notin I_n$, then let $J=P'+e^*$. It is easy to check that $J$ is an interval of $C$.

    We define the bijection $\sigma\colon E \to [n]$ as follows: $\sigma(e^*)=n$, and if $e=a_i$ or $e$ is the reverse of $a_i$  for some $i\in [n-1]$, then $\sigma(e)=i$. This is indeed a bijection, so it remains to show that 
    $e\in I_{\sigma(e)}\triangle  J$ for every $e \in E$. For $e=e^*$ this follows easily from the definition of $J$, because $e^*\in J$ if and only if $e^* \notin I_n$. If $e=a_i$ for some $i \in [n-1]$, then $e$ is in the clockwise paths of both $A_i$ and $A$, so $e \in I_i \setminus J$. If $e$ is the reverse of $a_i$ for some $i \in [n-1]$, then the reverse of $e$ is in the anticlockwise paths of both $A_i$ and $A$, so $e \in J \setminus I_i$.    
\end{proof}

\begin{rem}
The statement of Corollary~\ref{cor:representative} can be interpreted in terms of a matching on special bipartite graphs as follows.
Let $H = (\mathcal{I}, E; B)$ be a bipartite graph defined by $\mathcal{I} = \{ I_i \colon i \in [n] \}$ and $B = \{ (I_i, e) \colon e \in I_i \in \mathcal{I} \}$.
Then, $H$ is \emph{balanced} (i.e., $|\mathcal{I}| = |E|$), and since each $I_i$ is an interval, $H$ is  \emph{circular convex} (which is the definition of such a bipartite graph \cite{liang1995circular}). %
Also, let $H^c = (\mathcal{I}, E; B^c)$ be the balanced circular convex bipartite graph obtained as the complement of $H$, i.e., $B^c = (\mathcal{I} \times E) \setminus B$.
For a subset $E' \subseteq E$, let $H[E']$ and $H^c[E']$ denote the subgraphs of $H$ and of $H^c$, respectively, induced by $\mathcal{I} \cup E'$.
Under this rephrasing, the theorem claims that, for any balanced circular convex bipartite graph $H$ with any consistent cyclic ordering of $E$, there exists an interval $J \subseteq E$ (with respect to the cyclic ordering) such that the disjoint union of $H[J]$ and $H^c[E \setminus J]$ admits a perfect matching.

We also remark that we cannot replace ``cycle'' in the statement of Corollary~\ref{cor:representative} by ``path''. Indeed, if we consider a path $P=e_1e_2e_3$ of length 3, and $I_i=\{e_2\}$ for $i=1,2,3$, then there is no subpath $J$ of $P$ such that the family $I_1 \triangle J, I_2\triangle J, I_3 \triangle J$ has a system of distinct representatives. 
\end{rem}

\section{Proof of Theorem~\ref{thm:cycle}}
\label{sec:proof}

The proof strategy is very different from the proofs in the previous section. We consider the maximal directed rainbow paths on the cycle, and show that one of them can be extended to a rainbow spanning arborescence (Theorem \ref{thm:main2}). As observed in the previous section, once existence is known, one can test each vertex as a potential root in polynomial time to obtain a rainbow arborescence; hence we do not aim to provide an algorithmic proof, and our proof of existence is non-constructive. We assume for contradiction that no maximal rainbow path is extendible to a rainbow spanning arborescence, and examine the structural implications of this. On a very high level, the proof works as follows: we prove several structural inequalities based on our assumption, and get to a point where we can show that all of these inequalities must hold with equality. This puts severe restrictions on the possible configuration of the input arborescences, and after some additional work we can show that these restrictions cannot be satisfied at all. In the following, we describe the proof in detail, starting with the basic definitions and notation.

Let $(V,E)$ be a directed cycle of length $n$, where $V=\{v_1,\dots,v_n\}$, $e_j=v_jv_{j+1}$ ($j\in[n-1]$), and $e_n=v_nv_1$. We will usually consider the indices modulo $n$, i.e., $v_{n+j}=v_j$ and $e_{n+j}=e_j$. We denote the reverse arc of $e_j$ by $f_j$, that is, $f_j=v_{j+1}v_j$. The arcs $e_j$ will be called \emph{clockwise arcs}, while the arcs $f_j$ are \emph{anticlockwise arcs}. The graph $G= (V,E \cup F)$ is the bidirected cycle of length $n$, where $F = \{ f_j \colon j \in [n] \}$.
For a set $X$ of arcs, we denote by $\underlying{X}$ the set of corresponding edges in the underlying undirected graph; when $X = \{e\}$ for a single arc $e$, we denote by $\underlying{e}$ the singleton of the underlying edge of $e$ or the edge itself (depending on the context). A spanning arborescence $A$ of $G$ can be characterized by a pair of its root and the (undirected) edge that it does not use in either direction.
The latter will be called the \emph{missing edge} of $A$, which is sometimes referred to as the \textit{corresponding arc} of any direction.

   We use $A_1,\dots,A_{n-1}$ for the input arborescences.
   A clockwise path $P=(v_j,e_j,v_{j+1},\dots,e_{k-1},v_k)$ is called \emph{feasible} if there exists an injective function $c\colon\{j,\dots,k-1\} \to [n-1] $ such that $e_{\ell}\in A_{c(\ell)}$ for every $\ell \in \{j,\dots,k-1\}$.
Analogously, an anticlockwise path $Q=(v_j,f_{j-1},v_{j-1},\dots,f_{k},v_k)$ is \emph{feasible} if there exists an injective function $c\colon\{k,\dots,j-1\} \to [n-1] $ such that $f_{\ell}\in A_{c(\ell)}$ for every $\ell \in \{k,\dots,j-1\}$. For a clockwise path $P$, let $A(P)$ denote the unique spanning arborescence whose clockwise path is $P$. Similarly, for an anticlockwise path $Q$, let $A(Q)$ denote the unique spanning arborescence whose anticlockwise path is $Q$.

For an arc set $H \subseteq E \cup F$, let 
    \begin{align*}
    C(H) &\coloneqq \{c \in [n-1]\colon H \cap A_c \neq \emptyset\},\\
    \gamma(H) &\coloneqq |H|-|C(H)|.
    \end{align*}
We can observe that the set function $\gamma$ is supermodular. An arc set $B\subseteq E \cup F$ is a \emph{blocking set} if $\gamma(B)>0$. A blocking set $B$ is called a $\emph{clockwise blocking set}$ if $B\subseteq E$, an $\emph{anticlockwise blocking set}$ if $B \subseteq F$, and a $\emph{mixed blocking set}$ if $B \cap E \neq \emptyset$ and $B\cap F \neq \emptyset$.     
By Hall's theorem, a spanning arborescence $A$ of $G$ is rainbow if and only if there is no blocking set $B \subseteq A$. Furthermore, a clockwise path $P$ (respectively, anticlockwise path $Q$) is feasible if and only if there is no clockwise blocking set $B \subseteq P$ (respectively, no anticlockwise blocking set $B \subseteq Q$).
By the following lemma, we can assume that there exists no blocking singleton in any direction. 

We prove Theorem~\ref{thm:cycle} by induction on $n$.
The base case $n = 2$ is trivial, and in what follows we assume $n \ge 3$. 

\begin{lem}\label{lem:blocking_singleton}
    If there exists a blocking singleton in either direction, then the theorem holds. 
\end{lem}

\begin{proof}
By symmetry, suppose that there exists a clockwise blocking singleton $\{e_j = v_jv_{j+1}\}$.
If $e_j$ is the missing edge of all the colors, then the inclusionwise maximal clockwise feasible path $P$ ending at $v_j$ can be extended to the rainbow spanning arborescence $A(P)$ (just by adding anticlockwise arcs of the remaining colors greedily).

Otherwise, some color $c$ has the anticlockwise arc 
$f_j$.
In this case, we construct a smaller instance by removing the color $c$ and the vertex $v_j$, and by adding an anticlockwise arc $f' = v_{j+1}v_{j-1}$ to each of the remaining colors having both $f_{j-1}$ and $f_j$.
According to the induction hypothesis, the reduced instance has a rainbow spanning arborescence $A'$.
From $A'$, we can easily obtain a rainbow spanning arborescence $A$ of $G$ by adding $f_j \in A_c$, where if $A'$ contains the new arc $f'$ of color $c' \neq c$, then it should be replaced by $f_{j-1} \in A_{c'}$.
\end{proof}

From now on, we assume that there exists no blocking singleton in any direction.
Let $P_1,\dots,P_s$ be the set of inclusionwise maximal feasible clockwise paths (ordered according to the clockwise cyclic order of their first vertices), and let $Q_1,\dots,Q_t$ be the set of inclusionwise maximal feasible anticlockwise paths (ordered according to the anticlockwise cyclic order of their first vertices).
We consider the indices of those paths modulo $s$ and $t$, respectively.
We prove the following strengthening of Theorem \ref{thm:cycle} (under the assumption). 

\begin{thm} \label{thm:main2}
    Suppose that there exists no blocking singleton.
    Then, there exists $\ell \in [s]$ such that $A(P_{\ell})$ is rainbow, or there exists $\ell \in [t]$ such that $A(Q_{\ell})$ is rainbow.
\end{thm}
    
\begin{proof}
Since there exists no blocking singleton, each arc forms a feasible path of length $1$.
Thus, as each $P_\ell$ is maximal, we have $\bigcup_{\ell \in [s]} P_\ell = E$ and $s \ge 2$. 
We can also assume $|P_{\ell}|\leq n-2$ for every $\ell \in [s]$, because otherwise $P_{\ell}$ itself is a rainbow spanning arborescence (and we can make similar assumptions for $Q_1,\dots,Q_t$). Consider $P_{\ell}$ and $P_{\ell+1}$ for some index $\ell \in [s]$.
Let $e_j$ be the arc just after the last arc of $P_{\ell}$, and let $e_{j'}$ be the arc just before the first arc of $P_{\ell+1}$.
We then have $e_{j'} \in P_{\ell} \setminus P_{\ell + 1}$ and $e_j \in P_{\ell + 1} \setminus P_\ell$, and in particular, $e_j\neq e_{j'}$.

We show that there is a clockwise blocking set $X \subseteq \{e_{j'},\dots,e_j\}$ that contains both $e_{j'}$ and $e_j$. Indeed, $\{e_{j'},\dots,e_j\}$ cannot be a feasible path, since $P_{\ell}$ and $P_{\ell+1}$ were consecutive inclusionwise maximal feasible paths; this means that a clockwise blocking set $X \subseteq \{e_{j'},\dots,e_j\}$ must exist. This $X$ must contain both $e_{j'}$ and $e_j$, because otherwise $P_{\ell}$ or $P_{\ell+1}$ would not be feasible.

Let $X_{\ell}$ be an inclusionwise minimal blocking set with the above property. We denote by $\overline{X_{\ell}}$ the shortest subpath of $P_\ell+e_j$ that contains $X_{\ell}$; that is, $\overline{X_{\ell}}$ is the clockwise path from $e_{j'}$ to $e_j$.
We call $e_{j'}$ the \emph{first arc} of $X_{\ell}$, and call $e_j$ the \emph{last arc} of $X_{\ell}$. See Figure \ref{fig:xell} for an illustration.

\begin{figure}[h]
    \centering    
    \begin{subfigure}[t]{0.48\linewidth}
    \centering
    \includegraphics[width=0.7\linewidth]{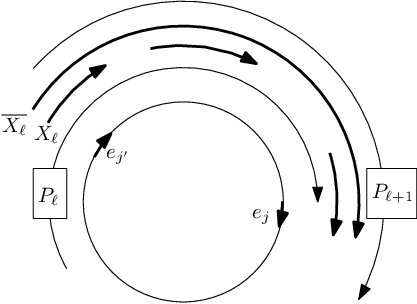}
    \caption{The positions of the arcs $e_j, e_{j'}$ and the arc sets $X_{\ell},\overline{X_{\ell}}$ with respect to the paths $P_{\ell}, P_{\ell+1}$.}
    \label{fig:xell_left}
    \end{subfigure}
    \hfill
    \begin{subfigure}[t]{0.48\linewidth}
        \centering
        \includegraphics[width=0.7\linewidth]{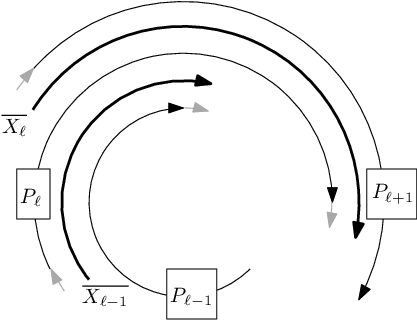}
    \caption{The positions of $X_{\ell-1}$ and $X_{\ell}$ with respect to the paths $P_{\ell-1}, P_{\ell}, P_{\ell+1}$ (the gray arcs are the first and last arcs of $X_{\ell-1}$ and $X_{\ell}$).}
    \label{fig:xell_right}
    \end{subfigure}
    \caption{Illustration of the definition of $X_\ell$ and $\overline{X_\ell}$.}
    \label{fig:xell}
\end{figure}

Using a similar argument for the anticlockwise paths $Q_{\ell}$ and $Q_{\ell+1}$, we can define an inclusionwise minimal anticlockwise blocking set $Y_{\ell}$ for $\ell \in [t]$.
We denote by $\overline{Y_{\ell}}$ the shortest subpath of $Q_\ell+f_j$ that contains $Y_{\ell}$, where $f_j$ is the arc after the last arc of $Q_{\ell}$.
We will sometimes refer to $X_\ell$ or $Y_\ell$ as another symbol, say $Z$, and then we also denote by $\overline{Z}$ the corresponding path $\overline{X_\ell}$ or $\overline{Y_\ell}$, respectively.
We also consider the indices of $X_\ell$ and $Y_\ell$ modulo $s$ and $t$, respectively, as with those of $P_\ell$ and $Q_\ell$.

We observe two properties on the positional relations of colors and blocking sets.

\begin{cl}\label{cl:disjoint}
  Suppose that there exists $A_c$ that is disjoint from both $X_{i}$ and $Y_{\ell}$. Then one of the following five possibilities holds:
  \begin{itemize}\itemsep0em
      \item $\underlying{\overline{X_i}}$ and $\underlying{\overline{Y_{\ell}}}$ are disjoint;
      \item $\underlying{\overline{X_i}} \cap \underlying{\overline{Y_{\ell}}}=\{e\}$, where $e$ underlies the first arcs of both $X_i$ and $Y_{\ell}$;
      \item $\underlying{\overline{X_i}} \subseteq \underlying{\overline{Y_{\ell}}}$, and the intersection of
      $\underlying{X_i}$ and $\underlying{Y_{\ell}}$ is at most one edge (which, if exists, underlies the first arc of $X_i$);
      \item $\underlying{\overline{X_i}} \supseteq \underlying{\overline{Y_{\ell}}}$, and  the intersection of
      $\underlying{X_i}$ and $\underlying{Y_{\ell}}$ is at most one edge (which, if it exists, underlies the first arc of $Y_\ell$);
      \item $\underlying{\overline{X_i}}$ and $\underlying{\overline{Y_{\ell}}}$ are co-disjoint, and the intersection of
      $\underlying{X_i}$ and $\underlying{Y_{\ell}}$ is at most one edge.
  \end{itemize}
\end{cl}
\begin{proof}
    If $|\underlying{X_i} \cap \underlying{Y_{\ell}}| \geq 2$, then there is an edge in the intersection that is not the missing edge of $A_c$, so $A_c$ intersects $X_i$ or $Y_{\ell}$, contradicting the assumption of the claim.

    Suppose that $|\underlying{X_i} \cap \underlying{Y_{\ell}}| \leq 1$, but none of the possibilities in the claim holds. 
Then, there are four indices $j_0, j_1, j_2, j_3$
such that
\begin{itemize}\itemsep0em
    \item $e_{j_0} \notin \overline{X_i}$, $f_{j_0} \notin \overline{Y_{\ell}}$
    \item $e_{j_1} \in \overline{X_i}$, $f_{j_1} \notin \overline{Y_{\ell}}$
     \item $e_{j_2} \notin \overline{X_i}$, $f_{j_2} \in \overline{Y_{\ell}}$
     \item $e_{j_3} \in \overline{X_i}$, $f_{j_3} \in \overline{Y_{\ell}}$, and $e_{j_3}$ is not the first arc of $X_i$ or $f_{j_3}$ is not the first arc of $Y_{\ell}$.
\end{itemize}

In this case, it can be checked by case analysis (see Figure \ref{fig:no_disjoint_arb}) that $A_c$ must contain at least one of the four (not necessarily distinct in the undirected sense) %
arcs given by the first and last arcs of $X_i$ and $Y_{\ell}$. 
\end{proof}

\begin{figure}[h]
    \centering
    \includegraphics[width=0.95\linewidth]{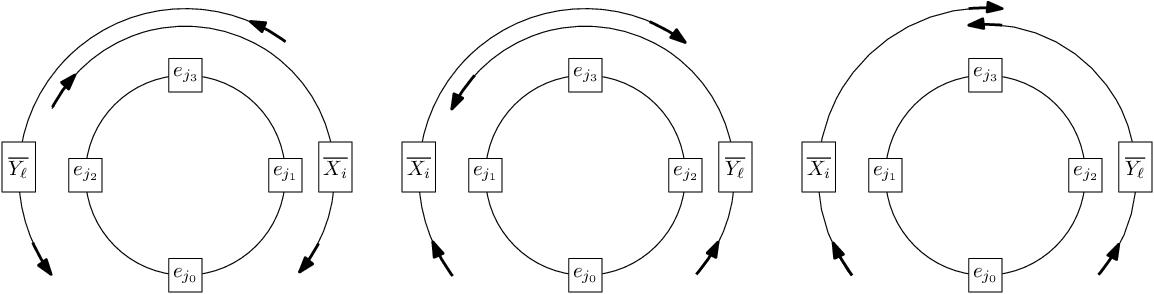}
    \caption{The three possible configurations of the first and last arcs of $X_i$ and $Y_{\ell}$ (in bold). In all cases, no spanning arborescence can be disjoint from all four of them.}
    \label{fig:no_disjoint_arb}
\end{figure}

\begin{cl}\label{cl:3blocking}
 If $\underlying{\overline{Y_{\ell}}}\subseteq \underlying{\overline{X_i}}$ and $|\underlying{\overline{Y_{\ell'}}} \cap \underlying{\overline{X_i}}|\leq 1$ for some indices $i,\ell,\ell'$, then every $A_c$ intersects at least one of $X_i,Y_{\ell}, Y_{\ell'}$. Similarly, if $\underlying{\overline{X_{\ell}}}\subseteq \underlying{\overline{Y_i}}$ and $|\underlying{\overline{X_{\ell'}}} \cap \underlying{\overline{Y_i}}|\leq 1$ for some indices $i,\ell,\ell'$, then every $A_c$ intersects at least one of $Y_i,X_{\ell}, X_{\ell'}$.
\end{cl}
\begin{proof}
    We prove the first statement; the proof of the second is analogous. Suppose that $A_c$ is disjoint from $X_i$ and $Y_{\ell}$. Since $\underlying{\overline{Y_{\ell}}}\subseteq \underlying{\overline{X_i}}$, the clockwise path of $A_c$ must be a subset of $\overline{X_i}$, and it cannot contain the first and last arcs of $X_i$; the latter implies that the missing edge of $A_c$ also belongs to $\underlying{\overline{X_i}}$. This implies that the anticlockwise path of $A_c$ contains all the reverse arcs of $E \setminus \overline{X_i}$, but then the properties $|\underlying{\overline{Y_{\ell'}}} \cap \underlying{\overline{X_i}}|\leq 1$ and $|Y_{\ell'}|\geq 2$ together imply that $Y_{\ell'}$ contains an anticlockwise arc of $A_c$.
\end{proof}

The next claims are consequences of the supermodularity of $\gamma$.

\begin{cl}\label{cl:capcup} $\gamma(X_{\ell})=1$,
  $\gamma(X_{\ell-1}\cap X_{\ell} )=0$ and $\gamma(X_{\ell-1}\cup X_{\ell} )=2$ for every $\ell \in [s]$. Similarly, $\gamma(Y_{\ell})=1$, $\gamma(Y_{\ell-1}\cap Y_{\ell} )=0$ and $\gamma(Y_{\ell-1}\cup Y_{\ell} )=2$ for every $\ell \in [t]$.
\end{cl}
\begin{proof}
    We prove the first statement; the proof of the second is analogous.
    Let $e_j$ be the first arc of $X_{\ell - 1}$, and let $e_{j'}$ be the last arc of $X_{\ell}$. 
    As $e_j$ is the arc just before the first arc of $P_\ell$ and $e_{j'}$ is the arc just after the last arc of $P_\ell$, the four sets $X_{\ell-1}-e_j$, $X_{\ell}-e_{j'}$, $X_{\ell-1}\cap X_{\ell}$, and $(X_{\ell-1}\cup X_{\ell}) \setminus\{e_j,e_{j'}\}$ are all included in $P_{\ell}$, so they are not blocking sets. It follows the definition of $\gamma$ that $\gamma(X_{\ell-1})=\gamma(X_{\ell})=1$, $\gamma(X_{\ell-1}\cap X_{\ell} ) \leq 0$, and $\gamma(X_{\ell-1}\cup X_{\ell} )\leq 2$. The supermodularity of $\gamma$ then implies the claim.
\end{proof}

\begin{cl}\label{cl:unionmu}
  Let $B_1,\dots,B_q$ be blocking sets such that $B_{i+1} \cap \left(\bigcup_{j=1}^{i} B_{j}\right)$ is not a blocking set for each $i\in [q-1]$. Then $\gamma\left(\bigcup_{i=1}^q B_{i}\right) \geq \sum_{i=1}^q \gamma(B_i)$.  
\end{cl}
\begin{proof}
    The proof is by induction on $q$, the case $q=1$ being trivial. Let $q \geq 2$, and let $B=\bigcup_{i=1}^{q-1} B_{i}$. Then $\gamma(B) \geq \sum_{i=1}^{q-1} \gamma(B_i)$ by induction hypothesis, and $\gamma(B_q \cap B)\leq 0$ by assumption. Therefore, the supermodularity of $\gamma$ implies that $\gamma\left(\bigcup_{i=1}^q B_{i}\right) \geq \sum_{i=1}^q \gamma(B_i)$.
\end{proof}

We will analyze the structure of the arborescences $A(P_{\ell})$ that are blocked by some anticlockwise blocking set. We start with a simple observation.

\begin{cl}\label{cl:Ki}
Suppose that $A(P_{\ell})$ is blocked by some anticlockwise blocking set. Then there exists an index $i$ such that $Y_i$ blocks $A(P_{\ell})$, and furthermore $\overline{Y_i}\subseteq A(P_{\ell})\setminus P_\ell$.     
\end{cl}

\begin{proof}
    Let $Q$ be the anticlockwise path of $A(P_{\ell})$.  
    Let $Q'$ be the longest feasible anticlockwise path ending at the last vertex of $Q$. Since $A(P_\ell)$ is blocked by some anticlockwise blocking set, $Q'$ cannot contain the first vertex of $Q$. There exists $i$ such that $Q_{i+1}$ contains $Q'$ and starts at the same vertex.  Furthermore, $Q_{i}$ cannot contain the last vertex of $Q$ due to the choice of $i$. Thus, $\overline{Y_i} \subseteq Q$ and $Y_i$ blocks $A(P_{\ell})$.
\end{proof}

Let $L$ be the set of indices $\ell$ such that $A(P_{\ell})$ is blocked by some anticlockwise blocking set. For $\ell \in L$, let $Z_{\ell}$ be the set $Y_i$ defined in the proof above 
(the proof also implies that $Z_\ell$ is the set $Y_i$ such that the anticlockwise path from the last vertex of $Y_i$ to the last vertex of $X_\ell$ is shortest). 
Note that $Z_{\ell}=Z_{\ell'}$ is possible for distinct indices $\ell,\ell' \in L$. Note also that $\underlying{\overline{X_{\ell}}}\cap\underlying{\overline{Z_{\ell}}}=\emptyset$, because $\underlying{{\overline{X_{\ell}}\setminus P_{\ell}}}$ is the missing edge of $A(P_{\ell})$, and 
$\overline{Z_{\ell}}\subseteq A(P_\ell)\setminus P_\ell$ by Claim \ref{cl:Ki}. The following lemma is one of the key components of the proof of Theorem \ref{thm:main2}.

\begin{lem}\label{lem:L}
    Let $L$ be defined as above, and let $X^*=\bigcup_{i=1}^s X_i$. Then $\gamma(X^*)=|X^*|-|C(X^*)| \geq |L|$. Furthermore, if $\bigcup_{\ell \in L} X_{\ell} \subsetneq X^*$, then $\gamma(X^*) > |L|$. 
\end{lem}
\begin{proof}
 The idea is to find a vertex $v_j$ that is not in the interior of any $\overline{X_i}$ $(i \in L)$. First, we show how the inequality in the lemma follows from the existence of such a $v_j$. We use Claim \ref{cl:unionmu} with the following parameters: $q=|L|$, and the sets $B_i$ are the sets $X_{\ell}$ ($\ell \in L$) in clockwise order, such that $v_j$ is between $\overline{B_q}$ and $\overline{B_1}$.
 It is easy to see that the conditions in the claim are satisfied due to the minimality of the blocking sets $X_{\ell}$, so we have
 \[\gamma\left({\textstyle\bigcup_{\ell \in L} X_{\ell}}\right)\geq \sum_{\ell \in L}\gamma(X_{\ell})=|L|.\]
 If $\bigcup_{\ell \in L} X_{\ell} \subsetneq X^*$, then we can greedily add additional sets $X_{\ell}$ ($\ell \notin L$) to the union, to finally obtain $\gamma(X^*)=|X^*|-|C(X^*)| > |L|$. 

 It remains to show that there exists a vertex $v_j$ that is not in the interior of any $\overline{X_i}$ $(i \in L)$.
 We may assume $|L| \ge 2$ (otherwise such a $v_j$ obviously exists).
 Let $\ell \in L$ be an index for which the clockwise path from the last vertex of $X_{\ell}$ to the last vertex of $Z_{\ell}$ is shortest. 
Then, let $\ell'\in L$ be the index minimizing the length of the anticlockwise path starting with the reverse of the first arc of $X_{\ell'}$ and ending with the first arc of $Z_{\ell}$ subject to either $\underlying{\overline{X_{\ell'}}}$ and $\underlying{\overline{Z_{\ell}}}$ are disjoint or their only common edge underlies the first arcs of both $X_{\ell'}$ and $Z_{\ell}$ (since $\underlying{\overline{X_\ell}}\cap \underlying{\overline{Z_\ell}}=\emptyset$ for every $\ell \in L$ and the first arcs of $X_i$ $(i \in [s])$ are distinct, such an $\ell' \in L$ uniquely exists). See Figure \ref{fig:mainlemma_a} for an illustration.

\begin{figure}[h]
    \centering
    \begin{subfigure}[t]{0.48\linewidth}
    \centering
    \includegraphics[width=0.55\linewidth]{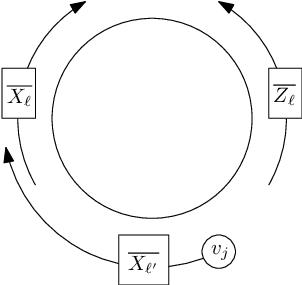}
    \caption{ $\underlying{\overline{Z_{\ell}}} \cap \underlying{\overline{X_{\ell'}}}=\emptyset$.}
    \label{fig:mainlemma_a_left}
    \end{subfigure}
    \hfill
    \begin{subfigure}[t]{0.48\linewidth}
    \centering
    \includegraphics[width=0.55\linewidth]{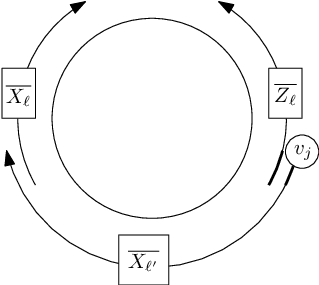}
    \caption{$\underlying{\overline{Z_{\ell}}} \cap \underlying{\overline{X_{\ell'}}}$ is a single edge that underlies the first arcs of both $Z_{\ell}$ and $X_{\ell'}$.}
    \label{fig:mainlemma_a_right}
    \end{subfigure}
    \caption{Possible relative positions of $\overline{Z_{\ell}}$ and $\overline{X_{\ell'}}$.}
    \label{fig:mainlemma_a}
\end{figure}

Let $v_j$ be the first vertex of $X_{\ell'}$; we show that $v_j$ satisfies the property that no $\overline{X_i}$ ($i \in L$) contains it in its interior, under the assumption that no other vertex satisfies this property.

Suppose to the contrary that $v_j$ is in the interior of $\overline{X_i}$ for some $i \in L$; we choose $i$ such that the clockwise path from the first vertex of $X_i$ to $v_j$ is shortest.
By the choice of $X_{\ell'}$, 
we have 
\begin{equation}
\text{$|\underlying{\overline{X_i}}\cap \underlying{\overline{Z_{\ell}}}|\ge 2$, and $\underlying{\overline{X_i}}$ and $\underlying{\overline{Z_{\ell}}}$ are not co-disjoint}, \label{eq:XiZl}
\end{equation}
where the latter follows from the fact that $\overline{X_{\ell}} \setminus \overline{X_i} \neq \emptyset$. We will show that, in any case, we obtain a contradiction.

We first consider the cases where $|\underlying{X_i} \cap \underlying{Z_{\ell}}| \leq 1$ and either $\underlying{\overline{Z_{\ell}}} \subseteq \underlying{\overline{X_i}}$ or $\underlying{\overline{X_i}} \subseteq \underlying{\overline{Z_{\ell}}}$; see Figure \ref{fig:mainlemma_b}.
Suppose first that $\underlying{\overline{Z_{\ell}}} \subseteq \underlying{\overline{X_i}}$. Then $|X_i|+|Z_{\ell}|+ |Z_{i}|  \leq n+1$ by the assumption that $|\underlying{X_i} \cap \underlying{Z_{\ell}}|\leq 1$, but no $A_c$ is disjoint from all three sets by Claim \ref{cl:3blocking}, so we obtain
\begin{equation}
 3 \leq \gamma(X_{i})+\gamma(Z_{\ell})+\gamma(Z_i)=|X_{i}|+|Z_{\ell}|+|Z_i|-|C(X_{i})|-|C(Z_{\ell})|-|C(Z_i)| \leq n+1-(n-1) =2, \label{eq:3sets}   
\end{equation}
a contradiction. Similarly, if $\underlying{\overline{X_i}} \subseteq \underlying{\overline{Z_{\ell}}}$, then $|X_{\ell}|+|Z_{\ell}|+|X_i|\leq n+1$ but no $A_c$ is disjoint from all of them by Claim \ref{cl:3blocking}, so we can get a contradiction in the same way.  

\begin{figure}[h]
    \centering
    \begin{subfigure}[t]{0.48\linewidth}
    \centering
    \includegraphics[width=0.55\linewidth]{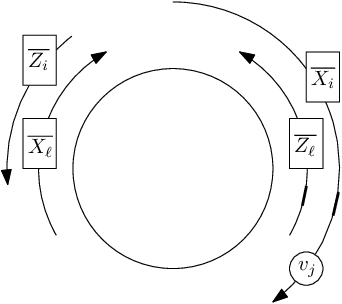}
    \caption{$\underlying{\overline{Z_{\ell}}} \subseteq \underlying{\overline{X_i}}$.}
    \label{fig:mainlemma_b_left}
    \end{subfigure}
    \hfill
    \begin{subfigure}[t]{0.48\linewidth}
    \centering
    \includegraphics[width=0.55\linewidth]{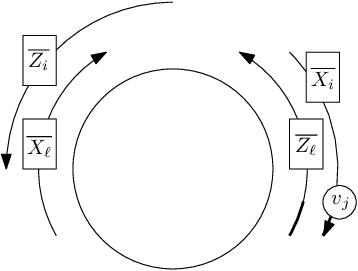}
    \caption{$\underlying{\overline{X_i}} \subseteq \underlying{\overline{Z_{\ell}}}$ (this is only possible if $v_j$ is the second vertex of $Z_{\ell}$).}
    \label{fig:mainlemma_b_right}
    \end{subfigure}
    \caption{Two easy cases where $\underlying{X_i}$ and $\underlying{Z_{\ell}}$ have at most one edge in common. }
    \label{fig:mainlemma_b}
\end{figure}

Thus, we can assume the following (two undirected paths are called incomparable if neither is a subpath of the other).
\begin{equation}
    \text{If $|\underlying{X_i} \cap \underlying{Z_{\ell}}| \leq 1$, then $\underlying{\overline{X_i}}$ and $\underlying{\overline{Z_{\ell}}}$ are incomparable.} \label{eq:incomparable}
\end{equation}

Next, we consider two cases based on the position of $\overline{Z_i}$. 

\paragraph{Case 1: $\underlying{\overline{X_{\ell}}} \cap \underlying{\overline{Z_i}} \neq \emptyset$, and the intersection is not a single edge that underlies the first arcs of both $X_{\ell}$ and $Z_i$.}  We have already seen in \eqref{eq:XiZl} that $|\underlying{\overline{X_i}}\cap \underlying{\overline{Z_{\ell}}}| \geq 2$,  and furthermore, they are not co-disjoint.
Note that $\underlying{\overline{X_{\ell}}}$ and $\underlying{\overline{Z_i}}$ are not co-disjoint either, because $\underlying{\overline{X_\ell}}$ and $\underlying{\overline{Z_i}}$ are both disjoint from $\underlying{\overline{Z_\ell}} \cap \underlying{\overline{X_i}} \neq \emptyset$ by definition. 
We consider subcases based on the relative positions of $\overline{X_{\ell}}$ and $\overline{Z_i}$; see Figure \ref{fig:mainlemma_bb}.

\begin{figure}[h]
    \centering
    \begin{subfigure}[t]{0.32\linewidth}
    \centering
    \includegraphics[width=0.75\linewidth]{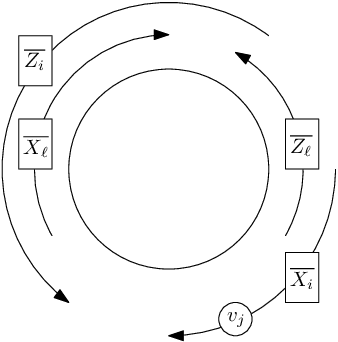}
    \caption{$\underlying{\overline{X_{\ell}}} \subseteq \underlying{\overline{Z_{i}}}$.}
    \label{fig:mainlemma_bb_left}
    \end{subfigure}
    \hfill
    \begin{subfigure}[t]{0.32\linewidth}
    \centering
    \includegraphics[width=0.75\linewidth]{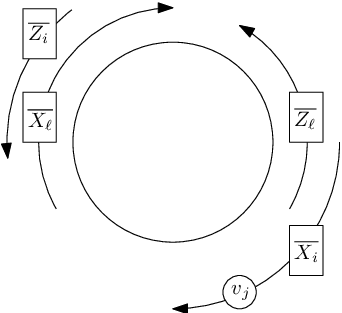}
    \caption{$\underlying{\overline{Z_{i}}} \subseteq \underlying{\overline{X_{\ell}}}$.}
    \label{fig:mainlemma_bb_center}
    \end{subfigure}
    \hfill
    \begin{subfigure}[t]{0.32\linewidth}
    \centering
    \includegraphics[width=0.75\linewidth]{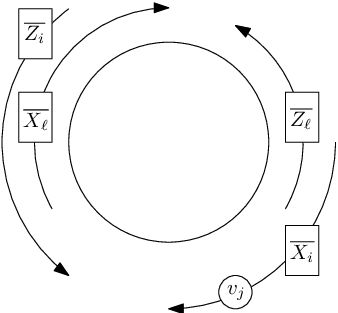}
    \caption{$\underlying{\overline{X_{\ell}}}$, $\underlying{\overline{Z_i}}$ are incomparable.}
    \label{fig:mainlemma_bb_right}
    \end{subfigure}
    \caption{Possible configurations of $\overline{X_{\ell}}$ and $\overline{Z_i}$ in Case 1.}
    \label{fig:mainlemma_bb}
\end{figure}

First, suppose that $|\underlying{X_{\ell}} \cap \underlying{Z_i}|\leq 1$ and either $\underlying{\overline{X_{\ell}}} \subseteq \underlying{\overline{Z_i}}$ or $\underlying{\overline{Z_{i}}} \subseteq \underlying{\overline{X_\ell}}$.
Then, as with the above argument concluding \eqref{eq:incomparable}, we get a contradiction by Claim~\ref{cl:3blocking}.

Thus, combining \eqref{eq:incomparable}, we may assume that \textit{i)} $|\underlying{X_{\ell}} \cap \underlying{Z_i}|\geq 2$ or $\underlying{\overline{X_{\ell}}}$ and $\underlying{\overline{Z_i}}$ are incomparable, and \textit{ii)} $|\underlying{X_i} \cap \underlying{Z_{\ell}}| \geq 2$ or $\underlying{\overline{X_i}}$ and $\underlying{\overline{Z_{\ell}}}$ are incomparable. Now, using Claim \ref{cl:disjoint}, \textit{i)} implies that each $A_c$ can be disjoint from at most one of $X_{\ell}$ and $Z_i$, and \textit{ii)} implies that each $A_c$ can be disjoint from at most one of $X_i$ and $Z_{\ell}$, so we obtain
\begin{multline*}
4 \leq \gamma(X_{\ell})+\gamma(Z_{\ell})+\gamma(X_i)+\gamma(Z_i) \\= |X_{\ell}|+|Z_{\ell}|+|X_i|+|Z_i|-|C(X_{\ell})|-|C(Z_{\ell})|-|C(X_i)|-|C(Z_i)|
\leq 2n - (2n-2) = 2,
\end{multline*}
a contradiction.

\paragraph{Case 2: $\underlying{\overline{X_{\ell}}} \cap \underlying{\overline{Z_i}}=\emptyset$, or their intersection is a single edge that underlies the first arcs of both $X_{\ell}$ and $Z_i$.} 
The last vertex of $Z_i$ cannot be closer to the last vertex of $X_{\ell}$ than the last vertex of $Z_{\ell}$ because of the definition of $Z_{\ell}$ 
(recall the remark just after Claim~\ref{cl:Ki}; $Z_\ell$ minimizes the length of the anticlockwise path from the last vertex of $Z_\ell$ to the last vertex of $X_\ell$).
Thus, the last vertices of $\overline{X_i}, \overline{Z_i}, \overline{X_\ell}$
are in this clockwise order, and $\overline{Z_i} \cap \overline{Z_{\ell}}=\emptyset$.
Also, the first vertex of $Z_i$ and the last vertex of $Z_\ell$ are distinct because the former must be on the anticlockwise path from the second vertex of $\overline{X_\ell}$ to the first vertex of $\overline{Z_\ell}$, while the latter is an interior of the complement of that path. See Figure \ref{fig:mainlemma_c}.

\begin{figure}[h]
    \centering
    \begin{subfigure}[t]{0.48\linewidth}
    \centering
    \includegraphics[width=0.55\linewidth]{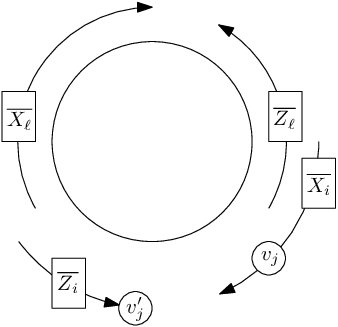}
    \caption{$\underlying{\overline{X_{\ell}}} \cap \underlying{\overline{Z_i}}=\emptyset$.}
    \label{fig:mainlemma_c_left}
    \end{subfigure}
    \hfill
    \begin{subfigure}[t]{0.48\linewidth}
    \centering
    \includegraphics[width=0.55\linewidth]{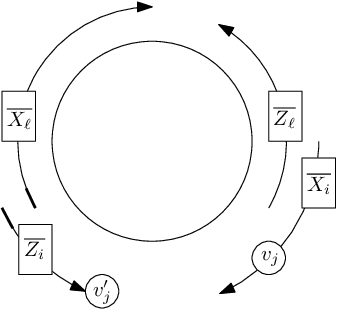}
    \caption{$\underlying{\overline{X_{\ell}}}\cap \underlying{\overline{Z_i}}$ is a single edge that underlies the first arcs of both $X_{\ell}$ and $Z_i$.}
    \label{fig:mainlemma_c_right}
    \end{subfigure}
    \caption{Possible configurations of $\overline{X_{\ell}}$ and $\overline{Z_i}$ in Case 2. 
    }
    \label{fig:mainlemma_c}
\end{figure}

Let $v_{j'}$ be the last vertex of $Z_i$.
We may assume that $v_{j'}$ is in the interior of $\underlying{\overline{X_{i'}}}$ for some $i' \in L$ (otherwise we are done).
Then, $i'\neq i$, and $|\underlying{\overline{X_{i'}}}\cap \underlying{\overline{Z_{\ell}}}|\leq 1$
by the definition of $\ell'$ and $i$ (and if they share an edge, then $i'=\ell' \neq \ell$). Note also that $\underlying{\overline{X_{i'}}}$ and $\underlying{\overline{Z_i}}$ are not co-disjoint, because $\overline{X_{i}} \not \subseteq \overline{X_{i'}}$. 
See Figure \ref{fig:mainlemma_cc} for an illustration.

\begin{figure}[h]
    \centering
    \begin{subfigure}[t]{0.48\linewidth}
    \centering
    \includegraphics[width=0.55\linewidth]{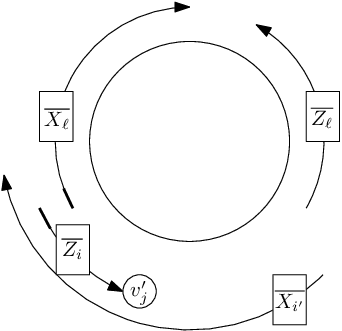}
    \caption{$\underlying{\overline{X_{i'}}} \cap \underlying{\overline{Z_{\ell}}}=\emptyset$.}
    \label{fig:mainlemma_cc_left}
    \end{subfigure}
    \hfill
    \begin{subfigure}[t]{0.48\linewidth}
    \centering
    \includegraphics[width=0.55\linewidth]{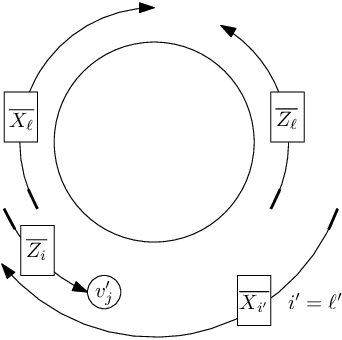}
    \caption{$\underlying{\overline{X_{i'}}}\cap \underlying{\overline{Z_{\ell}}}$ is a single edge that underlies the first arcs of both $X_{i'}$ and $Z_{\ell}$, and then $i'=\ell'$.}
    \label{fig:mainlemma_cc_right}
    \end{subfigure}
    \caption{Possible configurations of $\overline{X_{i'}}$ and $\overline{Z_{\ell}}$.}
    \label{fig:mainlemma_cc}
\end{figure}

We consider subcases based on the relationship of $X_{i'}$, $Z_i$, and $Z_{\ell}$; see Figure \ref{fig:mainlemma_d}.

\begin{figure}[h]
    \centering
    \begin{subfigure}[t]{0.32\linewidth}
    \centering
    \includegraphics[width=0.75\linewidth]{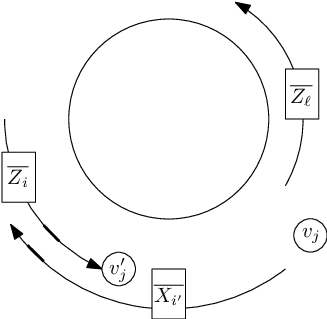}
    \caption{$|\underlying{X_{i'}} \cap \underlying{Z_i}|\leq 1$; $\underlying{X_{i'}} \cap \underlying{Z_{\ell}}=\emptyset$.}
    \label{fig:mainlemma_d_left}
    \end{subfigure}
    \hfill
    \begin{subfigure}[t]{0.32\linewidth}
    \centering
    \includegraphics[width=0.75\linewidth]{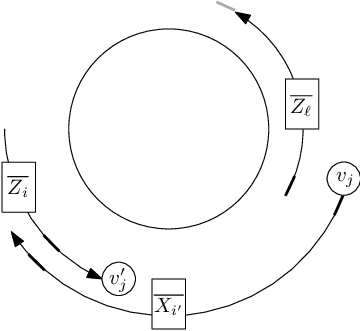}
    \caption{$|\underlying{X_{i'}} \cap \underlying{Z_i}|\leq 1$; the same edge underlies the first arcs of $X_{i'}$ and $Z_{\ell}$ (the gray edge is in neither).}
    \label{fig:mainlemma_d_center}
    \end{subfigure}
    \hfill
    \begin{subfigure}[t]{0.32\linewidth}
    \centering
    \includegraphics[width=0.75\linewidth]{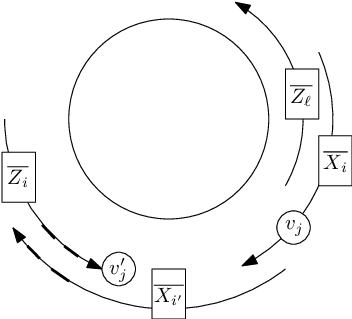}
    \caption{$|\underlying{X_{i'}} \cap \underlying{Z_i}|\geq 2$.}
    \label{fig:mainlemma_d_right}
    \end{subfigure}
    \caption{Subcases in Case 2.}
    \label{fig:mainlemma_d}
\end{figure}

First, consider the case where $|\underlying{X_{i'}} \cap \underlying{Z_i}|\leq 1$. Then $X_{i'}$, $Z_i$, and $Z_{\ell}$ have total size at most $n+1$ (recall $\overline{Z_i} \cap \overline{Z_{\ell}}=\emptyset$ and $|\underlying{X_{i'}}\cap \underlying{Z_{\ell}}|\leq 1$, and observe that if $|\underlying{X_{i'}}\cap \underlying{Z_{\ell}}|=1$, then $i' =\ell' \neq \ell$, so none of the three sets contains the edge after the last arc of $Z_{\ell}$ in the anticlockwise order). %
If in addition $\underlying{\overline{Z_i}} \subseteq \underlying{\overline{X_{i'}}}$, then every $A_c$ must intersect at least one of $X_{i'},Z_i,Z_{\ell}$ by Claim \ref{cl:3blocking}. If $\underlying{\overline{Z_i}} \not\subseteq \underlying{\overline{X_{i'}}}$, then none of the possibilities in Claim \ref{cl:disjoint} hold for $X_{i'}$ and $Z_i$ (note that they are not co-disjoint since neither of them contains the last edge of $\overline{Z_\ell}$), so every $A_c$ intersects at least one of them. Thus, we get a contradiction similar to \eqref{eq:3sets}:
\begin{equation*}
 3 \leq \gamma(X_{i'})+\gamma(Z_{\ell})+\gamma(Z_i)=|X_{i'}|+|Z_{\ell}|+|Z_i|-|C(X_{i'})|-|C(Z_{\ell})|-|C(Z_i)| \leq n+1-(n-1) =2. 
\end{equation*}

Now consider the subcase where $|\underlying{X_{i'}} \cap \underlying{Z_i}| \geq 2$; then, no $A_c$ can be disjoint from both $X_{i'}$ and $Z_i$ by Claim~\ref{cl:disjoint}. We also know by \eqref{eq:incomparable} that $|\underlying{X_i} \cap \underlying{Z_{\ell}}|\geq 2$ or $\underlying{\overline{X_i}}$ and $\underlying{\overline{Z_{\ell}}}$ are incomparable 
(note that, even in the latter case, we have $|\underlying{\overline{X_i}}\cap \underlying{\overline{Z_{\ell}}}| \geq 2$ and these two sets are not co-disjoint by \eqref{eq:XiZl}).

 By Claim~\ref{cl:disjoint}, each $A_c$ can be disjoint from at most one of $X_i$ and $Z_{\ell}$, and we have previously seen that each $A_c$ can be disjoint from at most one of $X_{i'}$ and $Z_i$.
Thus, as $\underlying{\overline{X_i}} \cap \underlying{\overline{Z_i}}=\emptyset$ and $|\underlying{\overline{X_{i'}}} \cap \underlying{\overline{Z_\ell}}|\leq 1$, we obtain
\begin{multline*}
4 \leq \gamma(X_{i'})+\gamma(Z_{\ell})+\gamma(X_i)+\gamma(Z_i) \\ = |X_{i'}|+|Z_{\ell}|+|X_i|+|Z_i|-|C(X_{i'})|-|C(Z_{\ell})|-|C(X_i)|-|C(Z_i)|
\leq 2n+1-(2n-2) =3,
\end{multline*}
a contradiction.
 This completes the proof that $v_j$ is not in the interior of any $\overline{X_i}$ $(i \in L)$.
\end{proof}

Recall that $X^*=\bigcup_{i=1}^s X_i$. Let $C^*=\{c \in C(X^*)\colon \exists i\in [s],\  \overline{X_i}\cap A_c=\emptyset\} ~(=C(X^*)\setminus (\bigcap_{i=1}^s C(\overline{X_i})))$. Our aim now is to show that $|C^*|\geq |X^*|-s$, which, together with Lemma \ref{lem:L}, will imply a strong structural property. We start with two preparatory claims.

\begin{cl}\label{cl:intersection}
    $\bigcap_{i=1}^s \overline{X_i}= \emptyset$. If $A_c$ intersects every $\overline{X_i}$, then it intersects at least one $X_i$.
\end{cl}
\begin{proof}
The second statement follows from the first, because if $A_c\cap \overline{X_i} \neq \emptyset$ but $A_c \cap X_i=\emptyset$, then $\overline{X_i}$ contains the clockwise path of $A_c$. We now prove the first statement. Let $e_j$ be arbitrary, and choose $i \in [s]$ such that $e_j \in P_i$, and among those, the clockwise path from $v_j$ to the last vertex of $P_{i}$ longest. If $e_{j+1}\notin P_{i+1}$, then $e_j \notin \overline{X_i}$, so we are done.

Assume that $e_{j+1} \in P_{i+1}$. By the choice of $i$ and the fact that $|P_{i+1}|\leq n-2$, we have $P_{i+1}\cap \{e_{j-1},e_j\}=\emptyset$. Then $e_j \notin \overline{X_{i+1}}$, and we are done. 
\end{proof}

An immediate consequence of Claim \ref{cl:intersection} is $X^* \subseteq \bigcup_{i=1}^s (\overline{X_i}\setminus \overline{X_{i-1}})$ (recall that we consider the indices modulo $s$, i.e., $X_0 = X_s$). Indeed, if $e \in X^*$, then there are indices $\ell$ and $\ell'$ such that $e \in \overline{X_{\ell}}\setminus \overline{X_{\ell'}}$, which implies that there is an index $i$ such that $e \in \overline{X_i}\setminus \overline{X_{i-1}}$.
\begin{cl}\label{cl:Xiprime}
    Let $i \in [s]$, and let $X_i'\coloneqq X^*\cap (\overline{X_i}\setminus \overline{X_{i-1}})$. Then $|C(X_i')\setminus C(\overline{X_{i-1}})|\geq |X_i'|-1$. 
\end{cl}
\begin{proof}
  Let $e_j$ be the first arc of $X_{i-1}$, and let $e_{j'}$ be the last arc of $X_i$. We know that $\gamma(X_{i-1})=1$, and also that $\gamma(X_{i-1}\cup X_i') \leq 2$, because $(X_{i-1}\cup X_i')-e_j-e_{j'} \subseteq P_i$. Thus, $|C(X_i')\setminus C(X_{i-1})|\geq |X_i'|-1$. The statement of the claim follows by observing that $C(X_i')\setminus C(X_{i-1})=C(X_i')\setminus C (\overline{X_{i-1}})$. This is because if $A_c \cap X_i'\neq \emptyset$ and $A_c \cap X_{i-1}= \emptyset$, then the root of $A_c$ is on the path $(\overline{X_i}\setminus \overline{X_{i-1}})-e_{j'}$, and the reverse of $\overline{X_{i-1}}-e_j$ is a subpath of the anticlockwise path of $A_c$; thus $A_c\cap \overline{X_{i-1}}=\emptyset$. See Figure \ref{fig:xiprime} for an illustration.
\end{proof}

  \begin{figure}[h]
    \centering
    \begin{subfigure}[t]{0.48\linewidth}
    \centering
    \includegraphics[width=0.55\linewidth]{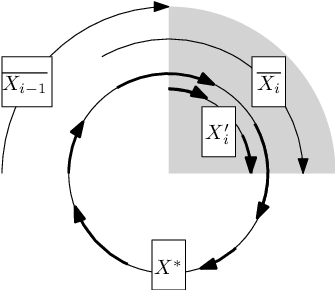}
    \caption{Illustration of the definition of $X_i'$.}
    \label{fig:xiprime_left}
    \end{subfigure}
    \hfill
    \begin{subfigure}[t]{0.48\linewidth}
    \centering
    \includegraphics[width=0.55\linewidth]{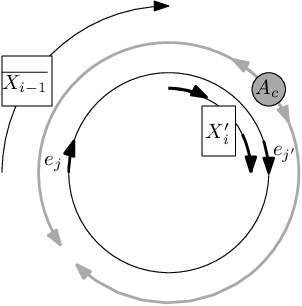}
    \caption{An arborescence $A_c$ for $c\in C(X_i')\setminus C(X_{i-1})$.}
    \label{fig:fig_xiprime_right}
    \end{subfigure}
    \caption{Illustration for Claim~\ref{cl:Xiprime}.}
    \label{fig:xiprime}
\end{figure}

Observe that $C(X_i')\setminus C(\overline{X_{i-1}}) \subseteq C^*$ for every $i \in [s]$. Furthermore, $C(X_i')\setminus C(\overline{X_{i-1}})$ is disjoint from $C(X_{\ell}')\setminus C(\overline{X_{\ell-1}})$ for every $\ell \neq i$, because the root of $A_c$ is a vertex of the path $\overline{X_i}\setminus \overline{X_{i-1}}$ which is not the last vertex of the path when $c \in C(X_i')\setminus C(\overline{X_{i-1}})$, and a vertex of the path $\overline{X_{\ell}}\setminus \overline{X_{\ell-1}}$ which is not the last vertex of the path when $c \in C(X_{\ell}')\setminus C(\overline{X_{\ell-1}})$; these cannot be the same.

By combining the above observations%
, Claim \ref{cl:Xiprime}, and $X^* \subseteq \bigcup_{i=1}^s (\overline{X_i}\setminus \overline{X_{i-1}})$ (shown just after Claim \ref{cl:intersection}),  
\begin{equation}
 |C^*| \geq \sum_{i=1}^s |C(X_i')\setminus C(\overline{X_{i-1}})|\geq \sum_{i=1}^s |X_i'|-s = |X^*|-s. \label{eq:Cstar} 
\end{equation}
By combining Lemma \ref{lem:L} with \eqref{eq:Cstar}, we get $s- |L| \geq |C(X^*)|-|C^*|$. 
By definition, the right-hand side is $|\{c \in C(X^*)\colon  A_c\cap\overline{X_i}\neq \emptyset, \ \forall i\in [s]\}|$, and hence, by using Claim \ref{cl:intersection}, this inequality is rewritten as follows:

\begin{equation} \label{eq:bound1}
 s-|L| \geq |\{c \in [n-1]\colon  A_c\cap \overline{X_i} \neq \emptyset, \ \forall i\in [s]\}|.
\end{equation}
 Let $L'$ denote the set of indices $\ell$ such that $A(Q_{\ell})$ is blocked by some clockwise blocking set. By a similar argument as above for the paths $Q_i$, we get
\begin{equation}\label{eq:bound2}
    t-|L'| \geq |\{c \in [n-1]\colon  A_c\cap \overline{Y_i} \neq \emptyset, \ \forall i\in [t]\}|.
\end{equation}

The following lemma implies that if Theorem \ref{thm:main2} fails to hold, then these inequalities must be tight. %

\begin{lem}\label{lem:bound}
If Theorem \ref{thm:main2} fails to hold, then 
\begin{align*}
  s-|L| &\leq |\{c \in [n-1]\colon  A_c\cap Y_i \neq \emptyset, \ \forall i\in [t]\}|,\\
  t-|L'| &\leq |\{c \in [n-1]\colon  A_c\cap X_i \neq \emptyset, \ \forall i\in [s]\}|.
\end{align*}   
\end{lem}

\begin{proof}
  It is enough to prove the first inequality. If Theorem \ref{thm:main2} does not hold, then  $A(P_{\ell})$ must be blocked by some mixed blocking set $M_{\ell}$ for each $\ell \in [s]\setminus L$. 
  
  \begin{cl}\label{cl:MellPell}
      We can choose $M_{\ell}$ so that $M_{\ell} \cap E=P_{\ell}$.
  \end{cl}
\begin{proof}
 First, observe that $M_{\ell} \cap E$ can be assumed to be a subpath of $P_{\ell}$, because if $M_{\ell}'$ is obtained from $M_{\ell}$ by extending $M_{\ell} \cap E$ to a shortest subpath of $P_{\ell}$, then $C(M_{\ell}')=C(M_{\ell})$, and hence $M_{\ell}'$ is also a blocking set that blocks $A(P_{\ell})$. (To see $C(M_{\ell}')=C(M_{\ell})$, observe that $M_{\ell}$ contains at least one arc of the anticlockwise path of $A(P_{\ell})$ because it is a mixed blocking set. Therefore, any $A_c$ that is disjoint from $M_{\ell}$ is also disjoint from $M_{\ell}'$.)

Suppose now that $M_{\ell} \cap E$ is a proper subpath of $P_{\ell}$, and the $M_{\ell}''$ obtained by adding the remaining arcs of $P_{\ell}$ is not a blocking set. 
 Then $|C(M_{\ell}'')\setminus C(M_{\ell})|>|M_{\ell}''\setminus M_{\ell}|$.
  
  Consider $X\coloneqq (X_{\ell-1} \cup X_{\ell})\cap M_{\ell}$, which is a subset of $P_\ell$ (see Figure \ref{fig:Mell}). We know that $\gamma(X_{\ell-1} \cup X_{\ell})=2$ by Claim \ref{cl:capcup}. On the one hand, $|X|\geq |(X_{\ell-1} \cup X_{\ell})|-|M_{\ell}''\setminus M_{\ell}|-2$. On the other hand, we show that if $c \in C(M_{\ell}'')\setminus C(M_{\ell})$, then $c \in C(X_{\ell-1} \cup X_{\ell})\setminus C(X)$. Since $C(X) \subseteq C(M_{\ell})$, $c \notin C(X)$ is obvious. To see $c \in C(X_{\ell-1} \cup X_{\ell})$, we again use that there exists an anticlockwise arc $f_j \in M_{\ell} \cap A(P_{\ell})$. Since $c\not\in C(M_\ell)$, we have $f_j \not\in A_c$. If $e_j \in A_c$, then the clockwise path of $A_c$ connects $e_j$ and $M_{\ell}''$, so it must go through the first arc of $X_{\ell-1}$ or the last arc of $X_{\ell}$ (see Figure \ref{fig:Mell}. If $e_j \notin A_c$, then its underlying edge is the missing edge of $A_c$; but then the clockwise path of $A_c$ must contain the last arc of $X_{\ell}$, since it contains a clockwise path from $M_{\ell}''$ to the missing edge.
  Thus $C(M_{\ell}'')\setminus C(M_{\ell}) \subseteq C(X_{\ell-1} \cup X_{\ell})\setminus C(X)$.
  We can conclude that
\begin{align*}
    \gamma(X) &= |X| - |C(X)|\\
    &\geq |(X_{\ell-1} \cup X_{\ell})|-|M_{\ell}''\setminus M_{\ell}|-2 - \left(|C(X_{\ell-1}\cup X_\ell)| - |C(M_{\ell}'')\setminus C(M_{\ell})|\right)\\
    &= \gamma(X_{\ell-1} \cup X_{\ell})-2+|C(M_{\ell}'')\setminus C(M_{\ell})|-|M_{\ell}''\setminus M_{\ell}|\\
    &> 0,
\end{align*}
which contradicts that $P_{\ell}$ is a feasible path.
\end{proof}

\begin{figure}[h]
    \centering
    \begin{subfigure}[t]{0.32\linewidth}
    \centering
    \includegraphics[width=0.75\linewidth]{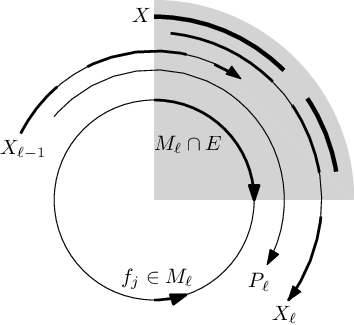}
    \caption{The definition of $X$.}
    \label{fig:fig_Mell_left}
    \end{subfigure}
    \hfill
    \begin{subfigure}[t]{0.32\linewidth}
    \centering
    \includegraphics[width=0.75\linewidth]{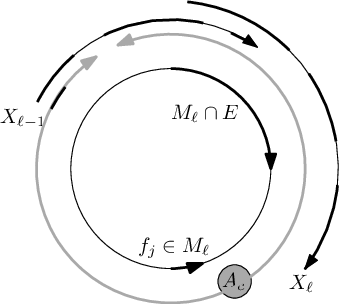}
    \caption{$c\in C(M_{\ell}'')\setminus C(M_{\ell})$, $A_c$ contains the first arc of $X_{\ell-1}$.}
    \label{fig:fig_Mell_right}
    \end{subfigure}
    \hfill
    \begin{subfigure}[t]{0.32\linewidth}
    \centering
    \includegraphics[width=0.75\linewidth]{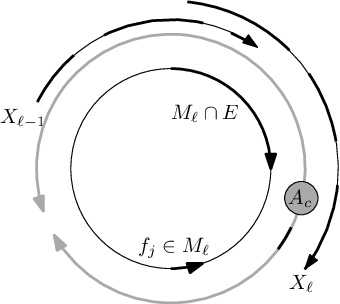}
    \caption{$c\in C(M_{\ell}'')\setminus C(M_{\ell})$, $A_c$ contains the last arc of $X_{\ell}$.}
    \label{fig:fig_Mell_center}
    \end{subfigure}
    \caption{Illustration for Claim~\ref{cl:MellPell}.}
    \label{fig:Mell}
\end{figure}

   In the following, we assume that $M_{\ell} \cap E=P_{\ell}$ is satisfied for every $\ell \in [s]\setminus L$. 

\begin{cl}\label{cl:Acproperties}
    Let $\ell \in [s]\setminus L$, and let $c\in [n-1]$ such that $A_c \cap M_{\ell}=\emptyset$. Then the reverse of $P_{\ell} \cup X_{\ell}$ is contained in $A_c$, and  $A_c \cap Y_i \neq \emptyset$ for every $i \in [t]$. 
\end{cl}
\begin{proof}
    Suppose for contradiction that the reverse of $P_{\ell}$ is not contained in $A_c$. Since $A_c$ is disjoint from $M_{\ell} \supseteq P_{\ell}$, this is only possible if the missing edge of $A_c$ underlies the first arc of $P_\ell$, which we denote by $e_j$. Furthermore, $A_c$ is not an anticlockwise path since it is disjoint from $M_{\ell}$. These imply that $X_{\ell-1}\cap A_c=\{e_{j-1}\}$, so $X_{\ell-1}-e_{j-1}$ is also a blocking set, which contradicts the feasibility of $P_{\ell}$.

    Let $e_{j'}$ be the last arc of $X_{\ell}$ (which is the unique arc in $X_\ell \setminus P_\ell$), and suppose $f_{j'} \notin A_c$. Then the root of $A_c$ is $v_{j'}$. Again, $A_c$ is not an anticlockwise path because it is disjoint from $M_{\ell}$, and $M_{\ell}$ contains an anticlockwise arc different from $f_{j'}$. These imply that $X_{\ell}\cap A_c=\{e_{j'}\}$, so $X_{\ell}-e_{j'}$ is also a blocking set, which contradicts the feasibility of $P_{\ell}$.
    
    Finally, consider any $Y_i$ $(i\in [t])$. We know that $Y_i \not \subseteq A(P_{\ell})$, because $\ell \notin L$. Thus, $Y_i$ contains an anticlockwise arc $f$ that is not in $A(P_{\ell})$, so $f$ is in the reverse of $P_{\ell}\cup X_{\ell}$. As we have proved above that $A_c$ contains the reverse of $P_{\ell}\cup X_{\ell}$, this shows $A_c \cap Y_i \neq \emptyset$.
\end{proof}

Now we are ready to prove the statement of the lemma. For $c\in [n-1]$, let $\alpha_c=|\{\ell\in [s]\setminus L\colon M_\ell \cap A_c = \emptyset\}|$. Let $C^+=\{c \in [n-1]\colon \alpha_c>0\}$. Take any $c\in C^+$. Then $A_c$ contains the reverse of $\alpha_c$ different paths $P_\ell\cup X_\ell$; let $E_c$ denote the union of these $\alpha_c$ paths, and let $F_c$ denote the reverse of $E_c$.
Let $\ell \in [s] \setminus L$ be any index such that $M_{\ell} \cap A_c = \emptyset$.
Then, $M_{\ell}\cap F_c=\emptyset$ as $F_c \subseteq A_c$.
We have also assumed that $M_{\ell}\cap E=P_{\ell}$, so the arcs of $E_c\setminus P_{\ell}$ are not in $M_{\ell}$.
Since $|E_c\setminus P_{\ell}|\geq \alpha_c$, we get that $|M_{\ell}|\leq n-\alpha_c$.

For $\ell \in [s]\setminus L$, let $\beta_{\ell}=|\{c \in C^+\colon M_{\ell} \cap A_c=\emptyset\}|$. Note that $\beta_{\ell}=n-1-C(M_{\ell})$. Since $M_{\ell}$ is a blocking set, we have $\beta_{\ell} \geq n-|M_{\ell}|$. Combining this with the previous inequality, we obtain that
$\alpha_c \leq \beta_{\ell}$ whenever $A_c \cap M_{\ell}=\emptyset$. From this, we get
\begin{align}
   s-|L|
   &= \sum_{\ell \in [s]\setminus L} \frac{1}{\beta_{\ell}}|\{c \in C^+\colon M_{\ell} \cap A_c=\emptyset\}| \nonumber \\ &= \sum_{c \in C^+} \sum_{\substack{\ell \in [s]\setminus L:\\ M_{\ell} \cap A_c=\emptyset}} \frac{1}{\beta_{\ell}} \nonumber \\ &\leq \sum_{c \in C^+} \sum_{\substack{\ell \in [s]\setminus L:\\ M_{\ell} \cap A_c=\emptyset}} \frac{1}{\alpha_c} \nonumber \\ &= \sum_{c \in C^+} \frac{1}{\alpha_c} |\{\ell\in [s]\setminus L\colon M_\ell \cap A_c = \emptyset\}| \nonumber \\ &=|C^+| \nonumber \\ &\leq |\{c \in [n-1]\colon  A_c\cap Y_i \neq \emptyset, \ \forall i\in [t]\}|, \label{eq:cplus}
\end{align}
where the last inequality follows from Claim~\ref{cl:Acproperties}.
The lemma's second inequality can be proved similarly.
\end{proof}

By combining inequalities \eqref{eq:bound1} and \eqref{eq:bound2} with Lemma \ref{lem:bound}, we get that, if Theorem \ref{thm:main2} fails to hold,  all four quantities in Lemma \ref{lem:bound} must be equal, and furthermore, all estimations that we used in the proofs of the inequalities and the lemmas must be tight. To conclude the proof of Theorem \ref{thm:main2}, we have to show that this is impossible. First, we show some implications of the tightness of the inequalities. In addition to the notation already introduced, including inside of the proof of Lemma~\ref{lem:bound} (e.g., $M_\ell$, $E_c$, $F_c$, $\alpha_c$, and $\beta_\ell$), we define $Y^*=\bigcup_{i=1}^t Y_i$.

\begin{lem}\label{lem:tight}
    If Theorem \ref{thm:main2} fails to hold, then the following are true:
    \begin{enumerate}[label = {\rm(\roman*)}]\itemsep0em
        \item $s-|L|=t-|L'|$; \label{item1}
        \item $X^*=\bigcup_{\ell \in L} X_{\ell}$ and $Y^*=\bigcup_{\ell \in L'} Y_{\ell}$; \label{item2}
        \item $|L| \geq 2$ and $|L'| \geq 2$; \label{item3}
        \item $\gamma(X^*)=|L|$ and $\gamma(Y^*)=|L'|$; \label{item4}
        \item  If $A_c\cap \overline{Y_i} \neq \emptyset \ (\forall i\in [t])$, then there is $\ell \in [s]\setminus L$ such that $A_c\cap M_{\ell}= \emptyset$. \label{item:Ac}
    \end{enumerate}
\end{lem}
\begin{proof}
  \ref{item1} holds because the four quantities in Lemma \ref{lem:bound} are equal, where
  \begin{align}
      |\{c \in [n-1]\colon  A_c\cap Y_i \neq \emptyset, \ \forall i\in [t]\}| &\le |\{c \in [n-1]\colon  A_c\cap \overline{Y_i} \neq \emptyset, \ \forall i\in [t]\}|,\label{eq:YiYibar}\\
      |\{c \in [n-1]\colon  A_c\cap X_i \neq \emptyset, \ \forall i\in [s]\}| &\le |\{c \in [n-1]\colon  A_c\cap \overline{X_i} \neq \emptyset, \ \forall i\in [s]\}|\nonumber,
  \end{align}
  hold with equality.
  The tightness of Lemma \ref{lem:L} implies \ref{item2} and \ref{item4}.
  To show \ref{item3}, observe that if $L=\{\ell\}$ and \ref{item2} hold, then the first arc of $X_{\ell-1}$ is in $X_{\ell}$, which is only possible if $P_{\ell}$ is a path of length $n-1$; but then $P_{\ell}$ is a rainbow arborescence.
  Finally, \ref{item:Ac} holds as the last inequality in \eqref{eq:cplus} and \eqref{eq:YiYibar} hold with equality.
\end{proof}
We consider two cases based on the sizes of $L$ and $L'$. We obtain contradictions in both cases under the assumption that Theorem \ref{thm:main2} fails to hold.

\paragraph{Case 1: $|L|=s$, $|L'|=t$.} Since $t-|L'|=0$ and \eqref{eq:bound2} holds with equality, each $A_c$ is disjoint from at least one $\overline{Y_i}$. Since $\gamma(X^*) = |L| \geq 2$ by \ref{item3} and \ref{item4} of Lemma \ref{lem:tight},
there exists $c$ such that $A_c \cap X^*=\emptyset$, and, by the above observation, there exists $i$ such that $A_c \cap \overline{Y_i}=\emptyset$. This is only possible if $u(\overline{Y_i})$ intersects $u(X^*)$ in at most one edge, and if it does, then this is the edge underlying the first arc of $Y_i$ and the missing edge of $A_c$. 

Let $f_j$ be the first arc of $Y_i$. Since $\bigcup_{i=1}^s P_i=E$, there is a path $P_{\ell}$ that contains the reverse of the last arc of $Y_i$. This implies that $P_{\ell}$ contains the reverse of $\overline{Y_i}-f_j$, because the arc after the last edge of $P_{\ell}$ is in $X^*$, while $u(X^*)$ and $u(\overline{Y_i}-f_j)$ are disjoint.
See Figure \ref{fig:final_case1}. %

\begin{figure}[h]
    \centering
    \begin{subfigure}[t]{0.48\linewidth}
    \centering
    \includegraphics[width=0.55\linewidth]{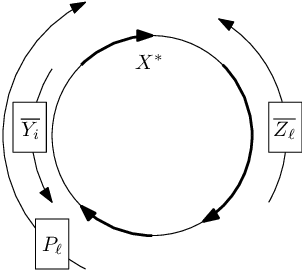}
    \caption{If $u(X^*) \cap u(\overline{Y_i})= \emptyset$, then $P_{\ell}$ contains the reverse of $\overline{Y_i}$.}
    \label{fig:fig_final_case1_left}
    \end{subfigure}
    \hfill
    \begin{subfigure}[t]{0.48\linewidth}
    \centering
    \includegraphics[width=0.55\linewidth]{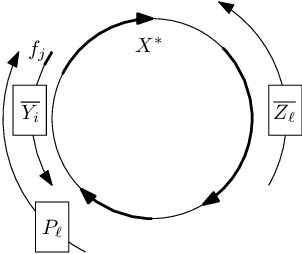}
    \caption{If $u(X^*) \cap u(\overline{Y_i})=\underlying{f_j}$, then $P_{\ell}$ contains the reverse of $\overline{Y_i}-f_j$.}
    \label{fig:fig_final_case1_right}
    \end{subfigure}
    \caption{Possible relative positions of $\overline{Y_i}$ and $P_{\ell}$ in Case 1.}
    \label{fig:final_case1}
\end{figure}

Since $\ell \in [s]=L$, there is an anticlockwise blocking set $Z_{\ell}$ that blocks $A(P_{\ell}$).  Let $X=X_{\ell-1} \cup X_{\ell}$. Then $\gamma(X)=2$ by Claim~\ref{cl:capcup}, so $\gamma(X)+\gamma(Y_i)+\gamma(Z_\ell)=4$. We know that $Z_\ell$ is disjoint from $Y_i$ (as $\overline{Z_\ell}\subseteq A(P_\ell)\setminus P_\ell$ by the definition of $Z_\ell$ and the reverse of $\overline{Y_i}-f_j$ is contained in $P_\ell$), while $u(X)$ intersects both $u(Y_i)$ and $u(Z_\ell)$ in at most one edge (because $u(X^*)$ and $u(\overline{Y_i}-f_j)$ are disjoint, and $X$ is contained in $P_\ell$ with one additional arc attached at both ends), so $|X|+|Y_i|+|Z_{\ell}| \leq n+2$.

Furthermore, we can show that no $A_{c'}$ is disjoint from all three of $X,Y_i,Z_{\ell}$ using a proof similar to that of Claim \ref{cl:3blocking}. Indeed, if $A_{c'}$ is disjoint from both $X$ and $Y_i$, then the clockwise path of $A_{c'}$ must be a subset of $P_{\ell}$, but then $Z_\ell$ contains an anticlockwise arc of $A_{c'}$.  Thus, we have
\[4=\gamma(X)+\gamma(Y_i)+\gamma(Z_{\ell})=|X|+|Y_i|+|Z_{\ell}|-|C(X)|-|C(Y_i)|-|C(Z_{\ell})|\leq n+2-(n-1)=3,\]
a contradiction.

\paragraph{Case 2: $s-|L|=t-|L'|>0$.} 
First, suppose that there exists $c \in [n-1]$ and $i \in [t]$ such that $A_c$ is disjoint from $X^* \cup \overline{Y_i}$. As in the proof of Case 1, this implies that $u(\overline{Y_i})$ intersects $u(X^*)$ in at most one edge, and if it does, then this is the edge underlying the first arc of $Y_i$ and the missing edge of $A_c$. 

Let $f_j$ be the first arc of $Y_i$. If there is an index $\ell \in L$ such that $P_{\ell}$ contains the reverse of $\overline{Y_i}-f_j$, 
then the same proof works as in Case 1. Therefore, we can assume that no such $\ell \in L$ exists. Let $\ell \in [s]\setminus L$ be an index such that $P_{\ell}$ contains the reverse of $\overline{Y_i}-f_j$ (such a path exists by the same argument as in Case~1), and subject to that, the path from the first vertex of $P_{\ell}$ to the last vertex of $Y_i$ is shortest. Let $e_{j'}$ be the first arc of $X_{\ell}$. Since $e_{j'}\in X^*$ while $u(X^*)\cap u(\overline{Y_i}-f_j)=\emptyset$, $e_{j'}$ is not on the reverse of $\overline{Y_i}-f_j$. In addition, the first vertex of $P_{\ell+1}$ cannot be on the clockwise path from the first vertex of $P_{\ell}$ to the last vertex of $Y_i$ by the choice of $\ell$. Then,  $e_{j'}$ 
is on the clockwise path from $v_j$ to the last vertex of $P_{\ell}$.
See Figure \ref{fig:final_case2a}.

\begin{figure}[ht]
    \centering
    \begin{subfigure}[t]{0.48\linewidth}
    \centering
    \includegraphics[width=0.55\linewidth]{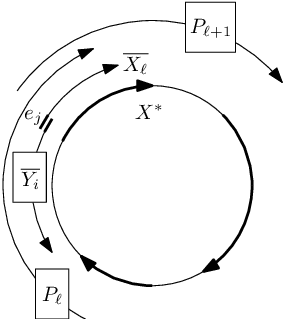}
    \caption{$e_{j'}=e_j$.}
    \label{fig:fig_final_case2a_right}
    \end{subfigure}
    \hfill
    \begin{subfigure}[t]{0.48\linewidth}
    \centering
    \includegraphics[width=0.55\linewidth]{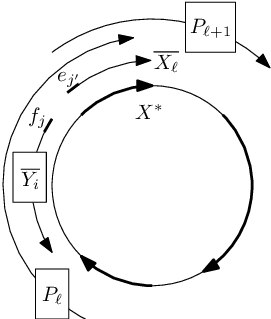}
    \caption{$e_{j'} \neq e_j$.}
    \label{fig:fig_final_case2a_left}
    \end{subfigure}
    \caption{The possible positions of $e_{j'}$ on the clockwise path from $v_{j}$ to the last vertex of $P_{\ell}$.}
    \label{fig:final_case2a}
\end{figure}

Since $X^*=\bigcup_{\ell' \in L} X_{\ell'}$ by \ref{item2} of Lemma~\ref{lem:tight}, there is an index $\ell' \in L$ such that $e_{j'}\in X_{\ell'}$. Since $P_{\ell'}$ does not contain the reverse of $\overline{Y_i}-f_j$ and we have $u(X^*)\cap u(\overline{Y_i}-f_j)=\emptyset$, the first vertex of $P_{\ell'}$ must be on the clockwise path from $v_{j+1}$ to $v_{j'}$.
But then it is in the interior of the clockwise path from the first vertex of $P_{\ell}$ to the first vertex of $P_{\ell+1}$, which is impossible.

We can therefore assume that there is no $c \in [n-1]$ and $i \in [t]$ such that $A_c$ is disjoint from $X^* \cup \overline{Y_i}$. Fix $c\in [n-1]$ such that $A_c$ is disjoint from $X^*$ (such a $c$ exists because $\gamma(X^*)=|L| \geq 2$ by Lemma \ref{lem:tight}). Then $A_c \cap \overline{Y_i} \neq \emptyset\ (\forall i\in [t])$ by our assumption, so $A_c$ is disjoint from $M_{\ell}$ for some $\ell \in [s]\setminus L$ because of property \ref{item:Ac} of Lemma \ref{lem:tight}.
There are $\alpha_c$ such indices $\ell \in [s] \setminus L$ by the definition of $\alpha_c$ given in the proof of Lemma~\ref{lem:bound}. Note that the tightness of \eqref{eq:cplus} implies the tightness of the inequalities used there, and hence $|E_c \setminus P_\ell| = \alpha_c = \beta_\ell$ holds for those $\alpha_c$ indices $\ell$, where we recall that $E_c$ is the union of the paths $P_\ell \cup X_\ell$ for the $\alpha_c$ indices.
These imply that the following hold for some $\ell^*$:
\begin{itemize}\itemsep0em
\item $\{\ell^*,\ell^*+1,\dots,\ell^*+\alpha_c-1\}\cap L=\emptyset$;
\item all the paths $P_\ell$ ($\ell\in \{\ell^*,\dots,\ell^*+\alpha_c-1\}$) are of the same length, and their first arcs form a path of length $\alpha_c$; 
\item the anticlockwise path of $A_c$ is exactly $F_c$, where recall that $F_c$ is the reverse of $E_c$. %
\end{itemize}
The first and second properties are immediate consequences of the equation $|E_c \setminus P_\ell| = \alpha_c$, and the third can be seen as follows.
Take any $\ell \in \{\ell^*,\dots,\ell^*+\alpha_c-1\}$.
By Claim~\ref{cl:Acproperties}, we have $F_c\subseteq A_c$, and hence $M_\ell \cap F_c \subseteq M_\ell \cap A_c = \emptyset$. We also have $M_\ell \cap E = P_\ell$ by Claim~\ref{cl:MellPell}. Then, $\underlying{E_c\setminus P_\ell}\cap \underlying{M_\ell}=\emptyset$ while $|E_c\setminus P_\ell|=\alpha_c$, which implies $|M_\ell|\leq n-\alpha_c$. Actually, this holds with equality because of the tightness in the proof of Lemma~\ref{lem:bound}  (more precisely, this follows from the tightness of $n - \beta_\ell \le |M_\ell| \le n - \alpha_c$). Then, we must have $M_\ell \cap F = F \setminus F_c$ to achieve this cardinality. Combined with $F_c\subseteq A_c$ and $M_\ell\cap A_c=\emptyset$, this concludes $A_c\cap F = F_c$.

Without loss of generality (by shifting the indices), we may assume that $\ell^* = 1$ in what follows.

Since $A_c \cap X^*=\emptyset$, every path $P_\ell$ ($\ell\in [s]$) is either (vertex-)disjoint from the clockwise path of $A_c$ or contains the clockwise path of $A_c$ (otherwise the first edge of $X_{\ell-1}$ or the last edge of $X_\ell$ belongs to $A_c$).
The above properties imply that if $\ell\notin [\alpha_c]$, then $P_\ell$ must contain the clockwise path of $A_c$.
In particular, this holds when $\ell\in L$.
Note also that if the clockwise path of $A_c$ is empty, then $P_\ell$ contains the missing edge of $A_c$ for every $\ell\notin [\alpha_c + 1]$.

Our aim now is to obtain a contradiction using $X^*=\bigcup_{\ell \in L} X_{\ell}$.
Let $\ell$ be the smallest index in $L \subseteq [s] \setminus [\alpha_c]$, and let $e_j$ be the first arc of $X_{\ell-1}$, which is the arc preceding the first arc of $P_{\ell}$.
As $\bigcup_{\ell' \in L} X_{\ell'} = X^*$, we have $e_j \in X_{\ell'}$ for some $\ell' \in L$ with $\ell' > \ell$, and both $P_{\ell}$ and $P_{\ell'}$ contain the clockwise path of $A_c$ (which may be empty); see Figure \ref{fig:final_case2b}.

\begin{figure}[ht]
    \centering
    \begin{subfigure}[t]{0.32\linewidth}
    \centering
    \includegraphics[width=0.75\linewidth]{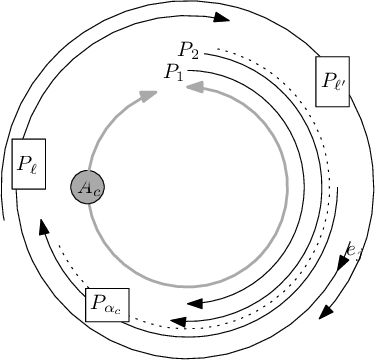}
    \caption{If $A_c \cap E \neq \emptyset$, then both $P_{\ell}$ and $P_{\ell'}$ contain the clockwise path of $A_c$.}
    \label{fig:fig_final_case2b_left}
    \end{subfigure}
    \hfill
    \begin{subfigure}[t]{0.32\linewidth}
    \centering
    \includegraphics[width=0.75\linewidth]{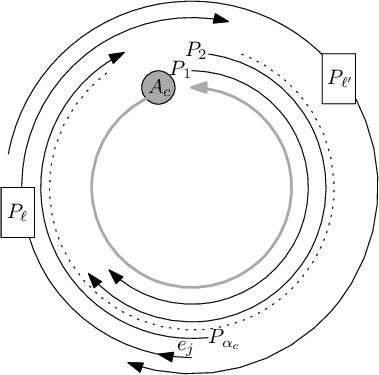}
    \caption{If $A_c \cap E = \emptyset$ and $\ell\neq\alpha_c + 1$, then both $P_{\ell}$ and $P_{\ell'}$ contain the missing edge of $A_c$.}
    \label{fig:fig_final_case2b_center}
    \end{subfigure}
    \hfill
    \begin{subfigure}[t]{0.32\linewidth}
    \centering
    \includegraphics[width=0.75\linewidth]{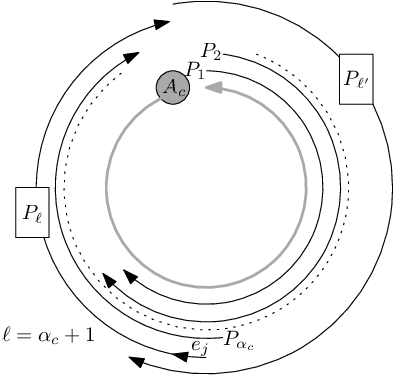}
    \caption{If $A_c \cap E = \emptyset$ and $\ell=\alpha_c + 1$, then it is possible that the missing edge of $A_c$ is in $P_{\ell'}\setminus P_{\ell}$.}
    \label{fig:fig_final_case2b_right}
    \end{subfigure}    
    \caption{The positions of paths with respect to $A_c$.}
    \label{fig:final_case2b}
\end{figure}

As $\ell\in L$, we have $\overline{Z_{\ell}} \subseteq A(P_{\ell})\setminus P_\ell$ (recall Claim~\ref{cl:Ki} and the definition of $Z_\ell$), and hence $\underlying{\overline{Z_{\ell}}} \subseteq \underlying{E\setminus P_\ell}$.
From this, $\underlying{\overline{Z_{\ell}}} \subseteq \underlying{\overline{X_{\ell'}}}$ follows.
To see this, observe $\overline{X_{\ell'}}$ includes the path $E\setminus P_\ell$ as follows: %
\begin{itemize}\itemsep0em
    \item The first arc of $\overline{X_{\ell'}}$ (i.e., the arc preceding $P_{\ell'+1}$) is either on $P_\ell$ or the arc succeeding $P_\ell$, because if $\ell'\neq s$, both $P_{\ell'+1}$ and $P_{\ell}$ contain the root of $A_c$ and $\ell' + 1 > \ell$, and otherwise $P_{\ell'+1}=P_1$ and its preceding arc is the missing edge of $A_c$.
    \item The last arc of $\overline{X_{\ell'}}$ is either on $P_\ell$ or is $e_j$, since $e_j\in X_{\ell'}$.
\end{itemize}
Thus, $\underlying{\overline{Z_{\ell}}} \subseteq \underlying{E\setminus P_\ell} \subseteq \underlying{\overline{X_{\ell'}}}$, and hence indeed $\underlying{\overline{Z_{\ell}}} \subseteq \underlying{\overline{X_{\ell'}}}$.
Note that this also implies $\underlying{\overline{Z_\ell}}\cap \underlying{\overline{Z_{\ell'}}}=\emptyset$ since $\underlying{\overline{X_{\ell'}}}\cap \underlying{\overline{Z_{\ell'}}}=\emptyset$ (as mentioned just after Claim~\ref{cl:Ki}).

Suppose first that $|\underlying{X_{\ell'}} \cap \underlying{Z_{\ell}}|\leq 1$. Then $|X_{\ell'}|+|Z_{\ell}|+|Z_{\ell'}|\leq n+1$, and every $A_{c'}$ intersects at least one of them by Claim \ref{cl:3blocking}, so  
\[3=\gamma(X_{\ell'})+\gamma(Z_{\ell})+\gamma(Z_{\ell'})\leq n+1-(n-1)=2,\]
a contradiction. Thus, we may assume that $|\underlying{X_{\ell'}} \cap \underlying{Z_{\ell}}|\geq 2$.

Next, we consider the blocking set $X\coloneqq X_{\ell-1}\cup X_{\ell}$; we use the notation $\overline{X}=X_{\ell-1}\cup X_{\ell} \cup P_{\ell}$.
Note that $\overline{X}$ is the clockwise path $P_\ell$ with its preceding arc $e_j$ and its succeeding arc attached.
We have $\gamma(X)=2$ by Claim \ref{cl:capcup}, and $\underlying{\overline{X}}$ is either disjoint from $\underlying{\overline{Z_{\ell}}}$ or their intersection is $\underlying{e_j}$ because $\overline{Z_{\ell}}\subseteq A(P_\ell)\setminus P_\ell$. Thus, $|\underlying{\overline{X}}\cap \underlying{\overline{Z_{\ell}}}|\leq 1$.
Furthermore, $\underlying{\overline{Z_{\ell'}}} \subseteq \underlying{\overline{X}}$ 
because $\underlying{\overline{Z_{\ell'}}} \subseteq \underlying{E\setminus P_{\ell'}}$ and the first arc of $\overline{X}$ (i.e., $e_j$) is either on $P_{\ell'}$ or the first arc of $E\setminus P_{\ell'}$ while the last arc of $\overline{X}$ (i.e., the arc succeeding $P_\ell$) is on $P_{\ell'}$ (recall that both $P_\ell$ and $P_{\ell'}$ contain the root of $A_c$ and $\ell < \ell'$).
If $|\underlying{X} \cap \underlying{Z_{\ell'}}|\leq 1$, then $|X|+|Z_{\ell}|+|Z_{\ell'}|\leq n+2$ as we have seen $\underlying{\overline{Z_\ell}}\cap \underlying{\overline{Z_{\ell'}}}=\emptyset$ and $|\underlying{\overline{X}}\cap \underlying{\overline{Z_{\ell}}}|\leq 1$. Also, every $A_{c'}$ intersects at least one of them, because if $A_{c'}$ is disjoint from both $X$ and $Z_{\ell'}$, then the clockwise path of $A_{c'}$ is a subset of $P_{\ell}$, so $Z_{\ell}$ contains an anticlockwise arc of $A_{c'}$. Thus, 
\[4=\gamma(X)+\gamma(Z_{\ell})+\gamma(Z_{\ell'})\leq n+2-(n-1)=3,\]
a contradiction. We can therefore assume that $|\underlying{X} \cap \underlying{Z_{\ell'}}|\geq 2$.
Then, every $A_{c'}$ intersects at least two of the four sets $X,Z_{\ell},X_{\ell'},Z_{\ell'}$ (one of $X,Z_{\ell'}$ and one of $Z_{\ell},X_{\ell'}$), which implies
\[5=\gamma(X)+\gamma(Z_{\ell})+\gamma(X_{\ell'})+\gamma(Z_{\ell'})\leq 2n+1-(2n-2)=3,\]
a contradiction. 
This final contradiction proves that Case 2 is impossible, which concludes the proof of Theorem \ref{thm:main2}, and also the proof of Theorem \ref{thm:cycle}.
\end{proof}

\section{Conclusion}
\label{sec:conc}

In this paper, we investigated the Rainbow Arborescence Conjecture, a special case of transversal-type problems arising from the Ryser--Brualdi--Stein conjecture in the setting of matroid intersection. We showed that finding a rainbow arborescence with a prescribed root is NP-complete, established the conjecture under several structural assumptions, and proved it for the important and nontrivial case when the underlying undirected graph is a cycle. The latter result also yields an application to systems of distinct representatives for families of intervals on a cycle. We close the paper with some open problems. 

\begin{enumerate}\itemsep0em
    \item The result of Kotlar and Ziv~\cite{kotlar2015rainbow} implies that if $G$ is the disjoint union of $n-1$ spanning arborescences, then it contains a rainbow branching of size $\lfloor n/2\rfloor$. In Theorem~\ref{thm:relax1}, we strengthened this by showing that the branching can, in fact, be chosen to be an arborescence. This naturally leads to the question of whether the recent theorem of Berger and McGinnis~\cite{berger2025common} admits a similar strengthening in the setting of arborescences: \emph{If $A_1,\dots,A_{n-1}$ are arborescences on $n$ vertices satisfying $|A_i|\ge \min\{i,\lfloor n/2\rfloor\}$ for each $i$, does there always exist a rainbow arborescence of size $\lfloor n/2\rfloor$?}
    \item A natural relaxation of Conjecture~\ref{conj:1} is to impose only proper coloring constraints. A directed analogue of the Maximum Properly Colored Forest problem was studied in~\cite{berczi2025approximating}, which considers two variants: in the \emph{properly out-colored} variant, only the outgoing edges of each vertex must have distinct colors, whereas in the \emph{properly colored} variant, all edges incident to a vertex -- both incoming and outgoing -- must have distinct colors. This leads to the following questions: \emph{If $A_1,\dots,A_{n-1}$ are spanning arborescences on $n$ vertices, does their union always contain a properly out-colored spanning arborescence? More strongly, does it always contain a properly colored spanning arborescence?}
\end{enumerate}

\paragraph{Acknowledgment.}
We are grateful to Shin-ichi Tanigawa for the initial discussion on the problem, which contains a key observation leading to Theorem~\ref{thm:two_multiroots}.

Yutaro Yamaguchi was supported by JSPS KAKENHI Grant Numbers 20K19743, 20H00605, and 25H01114, and by Start-up Program in Graduate School of Information Science and Technology, Osaka University.
Yu Yokoi was supported by JST ERATO Grant Number JPMJER2301. 
The research was supported by the Lend\"ulet Programme of the Hungarian Academy of Sciences -- grant number LP2021-1/2021, by the Ministry of Innovation and Technology of Hungary from the National Research, Development and Innovation Fund -- grant numbers ADVANCED 150556 and ELTE TKP 2021-NKTA-62, by Dynasnet European Research Council Synergy project -- grant number ERC-2018-SYG 810115, and by JST CRONOS Japan Grant Number JPMJCS24K2.

\bibliographystyle{abbrv}
\bibliography{rainbow}

\end{document}